\documentclass[12pt]{amsart}
\usepackage{amssymb,enumerate,mathrsfs, amsmath, amsthm}
\usepackage{url,titletoc,enumitem}

\usepackage[usenames,dvipsnames]{xcolor}
\definecolor{darkblue}{rgb}{0.0, 0.0, 0.55}
\usepackage[pagebackref,colorlinks,linkcolor=BrickRed,citecolor=OliveGreen,urlcolor=darkblue,hypertexnames=true]{hyperref}
\usepackage{cmap}

\usepackage{adjustbox}
\usepackage{graphicx}
\usepackage{float}
\usepackage{multirow}

\renewcommand{\subset}{\subseteq}

\newtheorem{theorem}{Theorem}[section]

\newtheorem{lemma}[theorem]{Lemma}

\newtheorem{prop}[theorem]{Proposition}
\newtheorem{algo}[theorem]{Algorithm}

\newtheorem{alg}[theorem]{Algorithm}
\newtheorem{remark}[theorem]{Remark}
\newtheorem{rem}[theorem]{Remark}

\newtheorem{obs}[theorem]{Observation}

\newtheorem{thm}[theorem]{Theorem}

\newtheorem*{lemma*}{Lemma}

\numberwithin{equation}{section}

\def\beq{\begin{equation}}
\def\eeq{\end{equation}}
\def\ben{\begin{enumerate}}
\def\een{\end{enumerate}}

\def\evidence{\noindent{ \bf Evidence:}\ }

\def\qed{$\blacksquare$}

\def\hbeta{\hat{\beta}}
\def\hgamma{\hat{\gamma}}

\def\cD{\mathcal{D}}

\def\cH{\mathcal{H}}

\def\cK{\mathcal{K}}

\def\hX{\hat{X}}

\def\C{\mathbb{C}}

\def\bbN{\mathbb{N}}

\def\bbR{\mathbb{R}}

\def\R{\mathbb{R}}

\def\smrg{{SM(\R)^g}}

\def\smnrg{{SM_n(\R)^g}}

\def\smdrg{{SM_d(\R)^g}}

\def\smnirg{{SM_{n_i}(\R)^g}}

\def\smxr[#1]{SM_{#1}(\R)}
\def\smxrg[#1]{SM_{#1}(\R)^g}

\newcommand{\df}[1]{{\bf{#1}}{\index{#1}}}

\def\abs{\partial^{\mathrm{free}}}
\def\free{\partial^{\mathrm{free}}}
\def\arv{\partial^{\mathrm{Arv}}}
\def\euc{\partial^{\mathrm{Euc}}}

\def\kAX{k_{A,X}}

\def\ncr{nearly certain}

\def\fep{\mbox{a free extreme point}}

\def\sec{\section}
\def\ssec{\subsection}
\def\sssec{\subsubsection}

\def\RC{RC}
\def\RPT{RPT}

\def\mathand{\quad \quad \mathrm{and} \quad \quad}

\title[Empirical properties of optima in free SDP]{Empirical properties of optima in free semidefinite programs}

\author[E. Evert]{Eric Evert${}^{1}$}
\address{Eric Evert, Group Science, Engineering and Technology\\
 KU Leuven Kulak \\
  E. Sabbelaan 53, 8500 Kortrijk, Belgium \\
  and
  \newline
   Electrical Engineering ESAT/STADIUS\\
  KU Leuven \\
  Kasteelpark Arenberg 10, 3001 Leuven, Belgium
   }
   \email{eric.evert@kuleuven.be}
\thanks{${}^1$Research supported by the NSF grant
DMS-1500835}

\author[Y. Fu]{Yi Fu${}^1$}
\address{Yi Fu, Department of Mathematics\\
  University of California \\
  San Diego
   }
   \email{yif064@ucsd.edu}
   
   \author[J.W. Helton]{J. William Helton${}^1$}
\address{J. William Helton, Department of Mathematics\\
  University of California \\
  San Diego
   }
   \email{helton@math.ucsd.edu }
   
   \author[J. Yin]{John Yin${}^1$}
\address{John Yin, Department of Mathematics\\
  University of Wisconsin -- Madison
   }
   \email{jbyin@wisc.edu }

\date{\today}

\makeindex

\begin{document}
	\maketitle

\begin{abstract}
Semidefinite programming is based on optimization of linear functionals over convex sets defined by linear matrix inequalities, namely, inequalities of the form
$$L_A(X)=I-A_1X_1-\dots-A_g X_g\succeq0.$$
Here the $X_j$ are real numbers and the set of solutions 
is called a spectrahedron.
These inequalities make sense when the $X_i$ are symmetric matrices of any size, $n\times n$,
and enter the formula though tensor product $A_i\otimes X_i$:
The solution set of $L_A(X)\succeq0$ is called a free spectrahedron
since it contains matrices of all sizes and the defining
``linear pencil" is ``free" of  the sizes of the matrices.

In this article, we report on empirically observed properties of optimizers obtained from optimizing linear functionals over free spectrahedra restricted to matrices $X_i$ of fixed size $n\times n$.

The optimizers  we find are always classical extreme points. Surprisingly, in many reasonable parameter ranges, over 99.9\% are also free extreme points. Moreover, the dimension of the active constraint, $\ker(L_A(X^\ell))$, is about twice what we expected. Another distinctive pattern regards reducibility of optimizing tuples $(X_1^\ell,\dots,X_g^\ell)$.

We give an algorithm for representing elements of a free spectrahedron as matrix convex combinations of free extreme points; these representations satisfy a very low bound on the number of free extreme points needed.

\end{abstract}

	\section{Introduction}

	\subsection{Context and motivation}  
One of the great engines of  math applied to technology
is linear programming, where one optimizes a linear function over a 
{\it polyhedron} of many (say g) real  numbers. 
About 25 years ago a powerful generalization of linear
programming called semidefinite programming emerged in
the convex optimization community and it has been vigorously
pursued since then (both in theory and applications).
Incidentally convex optimization is central to many areas of application.
Semidefinite programming pertains to optimizing a linear functional
over a convex set called a  {\it spectrahedron}.
Such sets are solution sets of
``linear matrix inequalities"  (LMI).

A sizeable community is   interested in studying such problems
when the unknowns are tuples of matrices and the 
structure of the LMI does not change with the size of the matrices. Such a problem is called ``dimension free".   Matrix multiplication does not commute
hence the term noncommutative (NC) LMI is often used.
This subject examines  {\it free spectrahedra} which
serve as a model for convex structures  that occur in
linear control and systems engineering problems
specified entirely by signal flow diagrams.

Pursuits like this  are in the spirit of the burgeoning area
 called {\it free analysis}. One of the original efforts here was Voiculescu's free probability,
which started by developing a theory for operator variables
and which has become a big area having many associations to
random matrix theory, \cite{MS17}.
Some other directions are  free
analytic function theory, cf. \cite{KVV14} and   free real algebraic geometry \cite{BKP16} \cite{N19}
with some consequences for system engineering being \cite{HMPV09}.

	\ssec{Free convex sets and free spectrahedra}

	Given a matrix $B \in \R^{n \times n}$ we say $B$ is \df{positive semidefinite} if $B$ is symmetric, i.e. $B=B^T$, and all the eigenvalues of $B$ are nonnegative. Let 
	\[
	B \succeq 0
	\]
	\index{$\succeq$}
	denote that the matrix $B$ is positive semidefinite. Similarly, given symmetric matrices $B_1, B_2 \in \R^{n \times n}$, let 
	\[
	B_1 \succeq B_2
	\]
	denote that the matrix $B_1-B_2$ is positive semidefinite.
	
	\sssec{Matrix Convex Combinations}
	
	Let $\smnrg$ \index{$\smnrg$} denote the set of $g$-tuples of real symmetric $n \times n$ matrices and set $\smrg=\cup_n \smnrg$ . That is, an element $X \in \smnrg$ is a tuple 
	\[
	X=(X_1, X_2, \dots X_g)
	\]
	where $X_i \in \R^{n \times n}$ and $X_i=X_i^T$ for each $i=1,2, \dots, g$. Similarly we let $M_{m \times n} (\R)^g$ \index{$M_{m \times n} (\R)^g$} denote the set of $g$-tuples of $m \times n$ matrices with real entries.
	
	Given a finite collection of tuples $\{X^i\}_{i=1}^\ell$ with $X^i \in \smnirg$ for each $i=1,\dots, \ell$, a \df{matrix convex combination of $\{X^i\}_{i=1}^\ell$} is a finite sum of the form
	\[
	\sum_{i=1}^\ell V_i^T X^i V_i \quad \qquad \mathrm{such \ that} \quad \qquad \sum_{i=1}^\ell V_i^T V_i=I_n.
	\]
Here $V_i \in M_{n_i \times n} (\R)$ and
	\[
	V_i^T X^i V_i=(V_i^T X_1^i V_i, V_i^T X_2^i V_i, \dots, V_i^T X_g^i V_i) \in \smnrg
	\]
	for each $i=1,2, \dots, \ell$. We emphasize that the matrix tuples $X^i$ can be of different sizes. That is, the $n_i$ need not be equal.
	
	As an example, if we take $g=2$ and $X=(X_1,X_2) \in \smxr[k]^2$ and $Y=(Y_1,Y_2) \in \smxr[m]^2$, then a matrix convex combination of $X,Y$ is a sum of the form
	\[
	V^TX V+W^TYW=(V^T X_1 V + W^T Y_1 W,\  V^T X_2 V+W^T Y_2 W)
	\]
	where $V:\R^n \to \R^k$ and $W:\R^n \to \R^m$ are contractions and $V^T V+W^T W=I_n$ for some positive integer $n$.

	A set $\Gamma \subset \smrg $ is \df{matrix convex} if it is closed under matrix convex combinations. The \df{matrix convex hull} of a set $\Gamma \subset \smrg$  is the set of all matrix convex combinations of elements of $\Gamma$. Equivalently, the matrix convex hull of $\Gamma$ is the smallest matrix convex set which contains $\Gamma$.

		\subsubsection{Free Spectrahedra and Linear Matrix Inequalities} In this article we focus on a class of matrix convex sets called free spectrahedra. A free spectrahedron is a matrix convex set which can be defined by a linear matrix inequality. Fix a $g$-tuple $A \in \smdrg$ of $d \times d$ symmetric matrices. A \df{monic linear pencil} $L_A (x)$ is a sum of the form 
\[
L_A (x)=I_d-A_1 x_1 - \dots - A_g x_g.
\index{$L_A (x)=I_d-A_1 x_1 - \dots - A_g x_g$}
\]
Given a tuple $X \in \smnrg$ the \df{evaluation} of $L_A$ on $X$ is
\[
L_A(X):= I_{dn}-A_1 \otimes X_1- \dots - A_g \otimes X_g
 \index{$L_A(X):= I_{dn}-A_1 \otimes X_1- \dots - A_g \otimes X_g$}
\]
where $\otimes$ denotes the Kronecker product. As an example, the \df{Kronecker product} of two matrices is
	\[
	\begin{pmatrix}
	1 & 2 \\
	3 & 4
	\end{pmatrix}
	\otimes B=
	\begin{pmatrix}
	1B & 2B \\
	3B & 4B
	\end{pmatrix}
	\]
	where the right hand side of the equality is a block matrix. A \df{linear matrix inequality} is an inequality with the form
\[
L_A(X) \succeq 0.
\]
Let $\Lambda_A(X)$ denote the linear part of the $L_A(X)$, i.e. 
\[
\Lambda_A(X):= A_1 \otimes X_1 + \dots +A_g \otimes X_g.
\]
\index{$\Lambda_A(X):= A_1 \otimes X_1 + \dots +A_g \otimes X_g$}
With this notation we have $L_A (X)=I_{dn}-\Lambda_A (X)$. 
	
	Given a $g$-tuple $A \in \smdrg$ and a positive integer $n$ we define the \df{free spectrahedron $\cD_A$ at level $n$}, denoted $\cD_A (n)$ \index{$\cD_A (n)$} by 
	\[
	\cD_A (n):= \{X \in \smnrg | \ L_A (X)=I_{dn}-A_1 \otimes X_1- \dots - A_g \otimes X_g \succeq 0 \}.
	\]
	Stated in words, the free spectrahedron $\cD_A$ at level $n$ is the set of $g$-tuples of $n \times n$ symmetric matrices $X$ such that the evaluation $L_A(X)$ is positive semidefinite. Define the \df{free spectrahedron} $\cD_A$ \index{$\cD_A$} to be the union over $n$ of the free spectrahedron $\cD_A$ at level $n$. In other words 
	\[
	\cD_A:= \cup_{n=1}^\infty \cD_A (n) \subset \smrg.
	\]
	It is not difficult to show that a free spectrahedron is matrix convex.
	
	\begin{lemma}
		Let $A \in \smdrg$ and let $\cD_A$ be the associated free spectrahedron. Then $\cD_A$ is matrix convex.
	\end{lemma}
	
	\begin{proof}
		Straightforward.
	\end{proof}
	
	\begin{remark}
		All matrix convex sets which satisfy the natural additional assumption that they are ``defined by a polynomial in noncommuting variables" are free spectrahedra \cite{HM12,HM14,K19}.
	\end{remark}
	
	A cultural remark is that when we restrict our attention to level $n=1$, sets $\cD_A (1)$ are precisely LMI representable sets with the cone of classical semidefinite programming.
	
	Say a free spectrahedron is \df{bounded} if there is a fixed real number $C$ so 
	\[
	CI_n-\sum_{i=1}^g (X_i)^2 \succeq 0
	\]
	for all $X=(X_1,\dots, X_g) \in \cD_A (n)$ and all positive integers $n$. It is straightforward to show that a free spectrahedron is bounded if and only if $\cD_A (1)$ is bounded. As a consequence of our use of nonstrict inequalities in the definition of a free spectrahedron, every free spectrahedron $\cD_A$ considered in this article is \df{closed} in the sense that $\cD_A (n)$ is closed for each integer $n$. However, other authors may consider strict inequalities when defining free spectrahedra.

	\subsection{Extreme Points of Free Spectrahedra} As with classical convex sets, there is much interest in the extreme points of matrix convex sets and free spectrahedra. We will consider two types of extreme points in this article: Euclidean (classical) extreme points and free (Arveson) extreme points.

	Given a matrix convex set $K$, say $X \in K(n)$ is a \df{Euclidean extreme point} of $K$ if $X$ is a classical extreme point of $K(n)$, i.e. if $X$ cannot be expressed as a nontrivial convex combination of elements of $K(n)$. We let $\euc K$ \index{$\euc K$} denote the set of Euclidean extreme points of $K$.

	Before introducing our second type of extreme point we give a brief definition. Given tuples $X,Y=(X_1, \dots, X_g) \in \smnrg$, if  there is an orthogonal (i.e. a real valued unitary) matrix $U$ so that
	\[
	U^T X U:=(U^T X_1 U, \dots, U^T X_g U)= (Y_1, \dots, Y_g)
	\]
	then we say $X$ and $Y$ are \df{unitarily equivalent}. We say a subset $E \subset \cD_A$ of a free spectrahedron is \df{closed under unitary equivalence} if $X \in E$ and $X$ is unitarily equivalent to $Y$ implies $Y \in E$. 
	
	Say $X \in K(n)$ is an \df{free extreme point} of $K$ if whenever $X$ is written as a  matrix convex combination 
	\[
	X=\sum_{i=1}^k V_i^T Y^i V_i \quad \qquad \mathrm{such\ that} \qquad \quad \sum_{i=1}^k V_i^T V_i =I_n
	\]
	with $V_i \neq 0$ for each $i$, then for all $i$ either $n_i=n$ and  $X$ is unitarily equivalent to $Y^i$ or $n_i > n$ and there exists a tuple $Z^i \in K$ such that  $X \oplus Z^i$ is unitarily equivalent to $Y^i$. We let $\abs K$ \index{$\abs K$} denote the set of free extreme points of $K$. Roughly speaking, a tuple is a free extreme point of a matrix convex set if it cannot be expressed as a nontrivial matrix convex combination of elements of the set.
	
	It is not difficult to show that a nontrivial convex combination of elements of a matrix convex set can be expressed as a nontrivial matrix convex combination. It follows that free extreme points are always Euclidean extreme points.
	  
	  	\sssec{Irreducibility of matrix tuples}
Free extreme points are irreducible as tuples of matrices, a notion we now define.	Given a matrix $M \in \R^{n \times n}$, a subspace $N \subset \R^n$ is a \df{reducing subspace} if both $N$ and $N^\perp$ are invariant subspaces of $M$. That is, $N$ is a reducing subspace for $M$ if $MN \subseteq N$ and $MN^\perp \subseteq N^\perp$. A tuple $X \in \smnrg$ is \df{irreducible} (over $\R$) if the matrices $X_1, \dots, X_g$ have no common reducing subspaces in $\R^n$; a tuple is \df{reducible} (over $\R$) if it is not irreducible (over $\R$).

Since we do not work over $\C$, we drop the distinction ``over $\R$" when discussing irreduciblity in the remainder of the article. However, we mention that irreducibility is frequently considered over $\C$ rather than over $\R$ in other articles. Irreducibility over $\C$ is equivalent to the matrices $X_1,\dots,X_g$ generating the full algebra of complex $n \times n$ matrices and also to the tuple $(X_1,\dots,X_g)$ having a trivial commutant. However, a tuple which is irreducible over $\R$ is not necessarily irreducible over $\C$  and may fail to generate the full matrix algebra. A tuple which is irreducible over $\R$ has a trivial symmetric commutant, but there may be skew symmetric matrices which commute with each element of the tuple. See Section \ref{sec:NumVerIrred}.

	\sssec{Free extreme points are the minimal spanning set of free spectrahedra}
	
	A central result in study of classical convex sets is the Minkowski Theorem which shows that any compact convex set in $\R^g$ is the convex hull of its of its extreme points. Furthermore, any subset of a convex set with this spanning property must contain the extreme points of the convex set. In this sense, free extreme points are the correct generalization of classical extreme points for a free spectrahedron.

	\begin{thm} \cite[Theorem 1.1]{EH19}
		\label{thm:AbsSpanMin}
		Let $A \in \smdrg$ such that $\cD_A$ is a bounded free spectrahedron. Then $\cD_A$ is the matrix convex hull of its free extreme points. 
	\end{thm}
	
 It is straightforward from the definition of free extreme points that if $E \subset \cD_A$ is a set of irreducible tuples which is closed under unitary equivalence and whose matrix convex hull is equal to $\cD_A$, then $E$ must contain the free extreme points of $\cD_A$. In this sense the free extreme points are the minimal set which recovers a free spectrahedron through matrix convex combinations.
 
	In addition to this qualitative statement there is a quantitative statement which serves as a natural extension of the Caratheodory theorem for free extreme points of a matrix convex set. For the free Caratheodory bound on the number of free extreme points needed to express a tuple $X$ in a free spectrahedron $\cD_A$ as a matrix convex combination of free extreme points, see Theorem \ref{thm:FreeCaratheodory}. The free Caratheodory bound is significantly less than the bound obtained from the classical Caratheodory theorem for expressing $X$ as a convex combination of classical extreme points of $\cD_A$.
	
	\ssec{Historical context for extreme points of matrix convex sets}
	Results for extreme points per se of matrix convex sets go back at least as far as \cite{WW99}. Here Webster and Winkler consider ``matrix" extreme points for matrix convex sets; a type of extreme point which is related too but less restrictive than free extreme points. Webster and Winkler showed that every compact matrix convex set is the closed matrix convex hull of its matrix extreme points. Matrix extreme points have been subsequently studied in various works, see for example \cite{F00,F04}, and the spanning result for matrix extreme points of Webster and Winkler was recently strengthen in \cite[Theorem 1.10]{HL+}.
	
	Although matrix extreme points have many desirable properties, matrix extreme points can fail to be a minimal spanning set for a matrix convex set. It is not difficult to construct a matrix convex set $K$ such that $K$ is the the matrix convex hull of its finitely many free extreme points (up to unitary equivalence) and such that $K$ has infinitely many (unitary equivalence classes of) matrix extreme points. To address this shortcoming and motivated by an infinite dimensional analogue, free extreme points were introduced in the setting matrix convex sets by Kleski in \cite{K14} and were subsequently studied in \cite{EHKM18} where they were called absolute extreme points. 
	
In this article we give little attention to matrix extreme points. Since we focus on free spectrahedra and since free spectrahedra are spanned by their free extreme points, this decision is well justified. However, free extreme points may fail to span general closed and bound matrix convex sets, see \cite{E18}. 
	
	
	
	\sssec{Free extreme points in infinite dimensions} Ironically, matrix convex sets and free extreme points first occurred (1960s) in the context of operators on an infinite dimensional Hilbert space, \cite{A69,A08}. Such ``operator convex sets" had an ``Arveson boundary" which much later was identified with free extreme points. The big open question was whether Arveson extreme points span. The big breakthrough occurred in \cite{DK15}; following on  \cite{DM05}, \cite{DK15} showed that if $X$ acts on a separable space $\cH$, then the free extreme points used to represent $X$ can be taken to act on a separable Hilbert space $\tilde{\cH}$. However, it remained unknown if $\tilde{\cH}$ could be taken to be finite dimensional if $\cH$ is finite dimensional. \cite{EH19}, stated here in Theorem \ref{thm:AbsSpanMin}, showed that what starts in finite dimensions stays in finite dimensions.
	
	\sssec{Curial matrix extreme points}
	We note that the even more restricted notion of ``crucial matrix extreme points" was recently introduced in \cite{PS19}. Roughly, crucial matrix extreme point fill the role of isolated extreme points for matrix convex sets. Crucial matrix extreme points and free extreme points are intended to serve different roles in the theory of extreme points for matrix convex sets. \cite[Proposition 2.6]{PS19} shows that crucial matrix extreme points of closed bounded matrix convex sets are necessarily free extreme points. Thus, based on above discussion, whenever the set of crucial matrix extreme points is a proper subset of the set of free extreme points, the crucial matrix extreme points cannot form a spanning set. Crucial matrix extreme points on the other hand are employed when studying a notion of the smallest operator tuple which defines a given matrix convex set. Also see \cite{DP20} for further discussion of crucial extreme points.

	\ssec{Main results and guide to the reader}
	
	Our results can be categorized in to two distinct flavors: theoretical and experimental, with experimental results receiving the most focus. Before discussing experimental results, a small collection of theoretical results is given in section in Section \ref{sec:Theory}. 
	
	Our theoretical results include formal solvability counts for the dimension of the kernel of $L_A (X)$ need for $X$ to be an (free or Euclidean) extreme point of a free spectrahedron $\cD_A$, see Theorems \ref{thm:BeanCount} and \ref{thm:BeanCountEuc}. We also improve on an algorithm originally presented in \cite{EH19} for representing an element of a free spectrahedron as a matrix convex combination of free extreme points, see Proposition \ref{prop:ArvDiAlg}.

	\sssec{Summary of experimental results}
	\label{sec:ExperimentalSummary}
	
	Our experiments generate extreme points of various free spectrahedra $\cD_A$ which are defined by irreducible $A$ by randomly generating and then optimizing linear functions over $\cD_A (n)$ for various choice of $n$. We then record and examine various properties of the optimizers. For details on how experiments were run see Section \ref{sec:ExperimentMethodology}.

The main properties of optimizers we consider are:

\ben
\item \label{it:FreeProportion} The proportion of irreducible optimizers which are free extreme points.
\item \label{it:KerDim} The dimension of the kernel of $L_A (\hat X)$ for irreducible optimizers $\hat X$, and the relationship between this kernel dimension and whether $\hat X$ is a free or Euclidean extreme point.
\item \label{it:IrredProp} The proportion of optimizers which are irreducible. 
\een

\paragraph{Free extreme proportion} Discussion of item \eqref{it:FreeProportion} takes place in Section \ref{sec:nonArvRare}. Here we present the very surprising data showing that an overwhelming majority of irreducible optimizers found in our experiments are free extreme points. Out of nearly 6,400,000 irreducible optimizers found in our experiments, only about  806 are not free extreme points (approx. $0.0126\%$). While one expects optimizers to be classical extreme points, we know no reason they should be free extreme points.

\paragraph{Dimension of $\ker L_A(\hat X)$} The dimension of $\ker L_A (\hat X)$ is primarily discussed in Sections \ref{sec:kerSize} and \ref{sec:KerDimDistri}. Section \ref{sec:kerSize} focuses on observed upper and lower bounds for the dimension of $\ker L_A (\hat X)$ for optimizers while Section \ref{sec:KerDimDistri} examines the distribution of observed kernel sizes.

  Theorems \ref{thm:BeanCount} and \ref{thm:BeanCountEuc} count, for fixed $g,d,n$ and dimension of $\ker L_A (\hat{ X})$, the number of equations and unknowns in certain homogeneous linear equations which govern whether the point $\hat{X} \in \smnrg$ is Euclidean or free extreme. In short, if the appropriate linear system has no nontrivial solutions, then the tuple of interest is Arveson or Euclidean extreme. Therefore, it is necessary that these solvability counts must be met for a tuple to be extreme. This gives a minimum size on $ \dim \ker L_A (\hat X)$ for $\hat X$ to be free extreme or Euclidean extreme. Dramatically, the minimum kernel dimension required for a tuple to be Arveson extreme is about twice as large as the minimum kernel dimension required for a tuple to be Euclidean extreme

By way of inequalities \eqref{eq:ArvKerCount} and \eqref{eq:EucKerCount}, the results presented in Section \ref{sec:kerSize} are consistent with the results Section \ref{sec:nonArvRare}. We find that the vast majority of optimizers $\hat X$ have $\dim \ker L_A (\hat X)$ greater than the lower bound given by inequality \eqref{eq:ArvKerCount} for a tuple to be a free extreme point. Additionally we find that $\dim \ker L_A (\hat X)$ is always nearly as big as this lower bound even if $\hat X$ is not a free extreme point of $\cD_A$.

Moreover, we find that for approximately 10\% percent of the irreducible extreme points found in our optimization experiments which are not free extreme, the dimension of $\ker L_A( \hat X)$  greater than or equal to the minimum bound given by \eqref{eq:ArvKerCount} needed for a tuple to be free extreme.  This important observation indicates that inequality \eqref{eq:ArvKerCount} does not (generically) serve as a sufficient condition for a tuple to be free extreme. In addition, there are large gaps between the smallest kernel sizes observed and the lower bound on the kernel size required for a tuple to be Euclidean extreme. Both these observations suggest that significant systematic degeneracy regularly occurs in the linear equations considered by inequalities \eqref{eq:ArvKerCount} and \eqref{eq:EucKerCount}.

Section \ref{sec:KerDimDistri} treats the dimension of the kernel of $L_A (\hat X)$ as a random variable and examines the distribution of the  observed kernel sizes. Our main finding in this section is that for spectrahedra defined by irreducible tuples in three or more variables, the distribution of kernel dimensions for optimizers generated by our experiments is well approximated by a Gaussian distribution. 

\paragraph{Proportion of optimizers which are irreducible} We find that the proportion of irreducible optimizers depends heavily on the method used to generate linear functionals. Two methods are used to randomly generate linear functionals, and linear functionals generated by these distinct methods are either called random coefficient (RC) or random positive trace (RPT) linear functions. The methods used are described in detail in Section \ref{sec:GenerationOfAell}.

In Section \ref{sec:irredExp} we examine the proportion of optimizers at level $n$ which are reducible as a function of $n$. We show evidence that when using RC linear functionals, this proportion is well fit by a decreasing exponential curve. However, when using RPT linear functionals, we find that this proportion can behave in a variety of ways and, notably, can increase as $n$ increases.

\paragraph{Software and data availability}
We provide the links to the NCSE package created for these experiments as well as the raw data produced by the experiment in Section \ref{sec:softwareanddata}.

\ssec{Acknowledgements}

We thank Maur\'icio de  Oliveira for development of NCSDP and other parts of NCAlgebra which were essential to the experiments in this paper and for helpful discussions. We also thank Tian Wu and Zinan Hu for helpful discussions related to this paper. In addition, we thank two anonymous referees for helpful comments.

	\section{Theory of Free Extreme Points underlying our experiments}

	\label{sec:Theory}
	
	This section  develops theory and
	gives some of the environment for our experiments. 
	It also contains a computational Caratheodory  algorithm. 
	We begin with essential definitions.

	\subsubsection{Minimal defining tuples} 
 	Given a tuple $A \in \smdrg$, we call $A$ a \df{defining tuple} for the free spectrahedron $\cD_A$. Any given free spectrahedron has infinitely many defining tuples. For example both $A$ and $A \oplus A$ are defining tuples for the free spectrahedron $\cD_A$. This leads to a small difficulty as in this article we wish to treat the size $d$ of the matrix tuple $A \in \smdrg$ which defines $\cD_A$ as a well-defined feature of $\cD_A$. 
 	
 	This difficulty may easily be overcome by introducing the notion of a \df{minimal defining tuple} for a free spectrahedron. Using \cite[Theorem 3.12 and Corollary 3.18]{HKM13} (also see \cite[Theorem 3.1]{Z17} for the unbounded case), we may simply define a minimal defining tuple for a free spectrahedron to be a tuple of minimal size that defines that the free spectrahedron. That is, if $A \in SM_{d_1} (\R^g)$ is a minimal defining tuple for the free spectrahedron $\cD_A$ and $B \in SM_{d_2} (\R^g)$ is any other defining tuple for $\cD_A$ then one must have $d_1 \leq d_2$. 
 	
 	\cite{HKM13} and \cite{Z17} show that any two minimal defining tuples of a given free spectrahedron are unitarily equivalent. Furthermore, it is shown that any defining tuple for a free spectrahedron must (up to unitary equivalence) contain a minimal defining tuple as a direct summand. 
 	
 	This has two consequence for our work. First, since we always use irreducible defining tuples in our experiments, our defining tuples are always minimal. Second, since the tuples we use are minimal defining tuples, the size of the tuples we use is indeed a well defined and intrinsic feature of the corresponding free spectrahedra.

\ssec{Free extreme points and the Arveson boundary} \ 
Free extreme points are closely related to the classical dilation theoretic Arveson boundary. Say a tuple $X$ in a matrix convex set $K$ is an \df{Arveson extreme point} of $K$ if
	\beq
	\label{eq:arvdef}
	Y= \begin{pmatrix} X & \beta \\
	\beta^T & \gamma \end{pmatrix}
	\in K
	\eeq
	implies $\beta=0$. The set of Arveson extreme points of a matrix convex set $K$ is called the \df{Arveson boundary} of $K$ and is denoted by $\arv K$ \index{$\arv K$}
	
	The following theorem relates the free, Arveson, and Euclidean extreme points of a free spectrahedron.
	
	\begin{thm}
		\label{thm:ArvImpliesEuc}
		Let $A \in \smdrg$ and let $\cD_A$ be the associated free spectrahedron.
		\ben 
		\item \label{it:FreeIsArvIrred} A tuple $X \in \cD_A (n)$ is a free extreme point of $\cD_A$ if and only if $X$ is an irreducible Arveson extreme point of $\cD_A$.  
		\item \label{it:ArvIsEuc} If $X \in \cD_A (n)$ is an Arveson extreme point of $\cD_A$, then $X$ is a Euclidean extreme point of $\cD_A$. 
		\een
	\end{thm}
	\begin{proof}

	\cite[Theorem 1.1]{EHKM18} proves \eqref{it:FreeIsArvIrred} and \eqref{it:ArvIsEuc} when working over $\C$. The \cite{EHKM18} proof of \eqref{it:ArvIsEuc} can be used over $\R$ without change. An adapted proof of \eqref{it:FreeIsArvIrred} which works over $\R$ is given by \cite[Theorem 1.2]{EH19}.
	\end{proof}
	
Given a free spectrahedron $\cD_A$, we say $X$ is a \df{non-Arveson extreme point} of $\cD_A$ if $X$ is a Euclidean extreme point of $\cD_A$ but $X$ is not an Arveson extreme point of $\cD_A$. While our focus in this article is on free extreme points, we use the terminology non-Arveson extreme point to emphasize that it is not a simple failure of irreducibility that prevents the tuple of interest from being free extreme.

	\sssec{Extreme points and linear systems}
	\label{sec:ExtLinSys}
	
	 	The classification of free extreme points as irreducible Arveson extreme points allows one to determine if a point is free extreme by solving a linear system. Given a free spectrahedron $\cD_A$ and a tuple $X \in \cD_A$ set
	\[
	\kAX:=\dim \ker L_A (X)
	\]
	and let $P_{A,X}: \R^{\kAX} \to \R^{nd}$ denote the inclusion map of $\R^{\kAX}$ onto the kernel of $L_A(X)$. In other words, $P_{A,X}$ is a matrix whose columns form an orthonormal basis for the kernel of $L_A(X)$.
	
	\begin{prop}
	\label{prop:ExtremeLinSys}
	Let $\cD_A$ be a free spectrahedron, and let $X \in \cD_A$. 
	\ben
	\item
	\label{it:ArvLinSystem}
	$X$ is an Arveson extreme point of $\cD_A$ if and only if the only solution to the linear system
	\beq
	\label{eq:ArvLinSystem}
	\Lambda_A (\beta^T) P_{A,X}=(A_1 \otimes \beta^T_1+ \dots + A_g \otimes \beta^T_g)P_{A,X}=0_{d \times \kAX}
	\eeq
	in the unknown $\beta \in M_{n \times 1} (\R)^g$ is $0 \in M_{n \times 1} (\R)^g$.
	\item
	\label{it:EucLinSystem}
	$X$ is an Euclidean extreme point of $\cD_A$ if and only if the only solution to the homogeneous linear system 
	\beq
	\label{eq:EucLinSystem}
	\Lambda_A (\beta^T) P_{A,X}=(A_1 \otimes \beta_1^T+ \dots + A_g \otimes \beta_g^T)P_{A,X}=0_{dn \times \kAX}
	\eeq
	in the unknown $\beta \in \smnrg$ is $0 \in \smnrg$
	\een
	\end{prop}
	
\begin{proof}
Items \eqref{it:ArvLinSystem} and \eqref{it:EucLinSystem} are immediate consequences of \cite[Lemma 2.1 (3)]{EH19} and \cite[Corollary 2.3]{EHKM18}, respectively. Also see \cite[Corollary 3]{RG95} for item \eqref{it:EucLinSystem}. For the reader's convenience, we outline a self contained proof. 

To prove item \eqref{it:ArvLinSystem}, set 
\[
Y= \begin{pmatrix}
X & \beta \\
\beta^T & \gamma
\end{pmatrix}.
\]
Conjugating by permutation matrices, sometimes called canonical shuffles (see \cite{P02}), shows that the evaluation $L_A (Y)$ is unitarily equivalent to the matrix
\[
\begin{pmatrix}
L_A (X) & -\Lambda_A (\beta) \\
-\Lambda_A (\beta^T) & L_A (\gamma) 
\end{pmatrix}.
\]
Taking the Schur complement of this matrix shows $Y \in \cD_A$, that is $L_A (Y) \succeq 0$, if and only if 
\beq
\label{eq:NCSchur}
\begin{split}
L_A (X) - \Lambda_A (\beta) \left(L_A (\gamma)\right)^\dagger \Lambda_A (\beta^T) \succeq 0 
\mathand
L_A (\gamma) \succeq 0 \\
\mathand
(I - L_A (\gamma) L_A(\gamma)^\dagger) \Lambda_A (\beta^T)=0.
\end{split}
\eeq
Here $\dagger$ denotes the Moore-Penrose pseudoinverse. 
It follows that if there is a nonzero $\beta$ such that $Y \in \cD_A$, that $\beta$ must satisfy
\beq
\label{eq:ExtremeKerContain}
\ker L_A (X) \subseteq \ker \Lambda_A (\beta^T),
\eeq
hence equation \eqref{eq:ArvLinSystem} as a nonzero solution. 

Conversely, if there there is a $\beta$ satisfying \eqref{eq:ExtremeKerContain}, then by taking $\gamma=0$ the argument above reverses to show $X$ is not Arveson extreme

The proof of item \eqref{it:EucLinSystem} follows a nearly identical strategy using \cite[Corollary 2.3 (iv)]{EHKM18} which shows that a tuple $X$ is a Euclidean extreme point of $\cD_A (n)$ if and only if, if $\beta \in \smnrg$ satisfies 
\[
\begin{pmatrix}
X & \beta \\
\beta & \gamma 
\end{pmatrix} \in \cD_A
\]
then $\beta=0$. 
\end{proof}	
	
	We emphasize the important distinction that $\beta$ is a tuple of symmetric matrices when working with Euclidean extreme points, but for Arveson extreme points, $\beta$ is simply a $g$-tuple of vectors. 
	
	These characterizations of extreme points are valuable in practice as they allow for numerical verification that a given tuple is a Arveson or Euclidean extreme point of a matrix convex set. One may further check if a given Arveson extreme point is free extreme (i.e. is irreducible) by determining the  symmetric commutant of the tuple, see Section \ref{sec:NumVerIrred}.
	
	\subsubsection{Kernel dimension counts corresponding to extremality}
	
 We now examine the dimension of the kernel of $L_A (X)$ which is required for $X$ to be an extreme point of $\cD_A$. In the following result we consider various amounts of irreducibility for the tuple $A$. For one we consider the case where a minimal\footnote{As a consequence of \cite[Theorem 3.1]{Z17}, the existence of any defining tuple of diagonal matrices implies the existence of a minimal defining tuple of diagonal matrices. However, a non-minimal defining tuple for a free polytope need not be simultaneously diagonalizable. } defining  tuple $A$ for the spectrahedron $\cD_A$ is simultaneously diagonalizable. In this case call $\cD_A$ a \df{free polytope}. We first consider Arveson extreme points.

	\begin{thm}
		\label{thm:BeanCount}
		Let $\cD_A \subset \smrg$ be a free spectrahedron where $A \in \smdrg$ is a minimal defining tuple for $\cD_A$, and let $X \in \cD_A (n)$.
		\begin{enumerate}
			\item \label{it:BeanCountIrred} If $X \in \arv \cD_A (n)$, then $gn \leq d\kAX$.
			\item \label{it:BeanCountPolytope}  If $\cD_A$ is a free polytope and $X \in \arv \cD_A$, then $gn \leq \kAX$.
			\item \label{it:BeanCountGeneral} Suppose $A$ is unitarily equivalent to the tuple $\oplus_{j=1}^\ell A^j$ where $A^j \in SM_{d_j} (\R)^g$ for each $j$. For $j=1, \dots, \ell$, set $k^j_{A,X}= \dim \ker L_{A^j} (X)$. If $X \in \arv \cD_A$, then
			\beq
			\label{eq:ArvKerCount}
			 gn \leq \sum_{j=1}^\ell d_j k^j_{A,X}.
			\eeq
		\end{enumerate}
		
	\end{thm}

	\begin{proof}
		
 We have assumed $A$ is a minimal defining tuple so items \eqref{it:BeanCountIrred} and \eqref{it:BeanCountPolytope} are special cases of item \eqref{it:BeanCountGeneral}. Therefore, it is sufficient to prove item \eqref{it:BeanCountGeneral}.
		
		To this end assume $A$ is unitarily equivalent to the tuple $\oplus_{j=1}^\ell A^j$ where each  $A^j \in SM_{d_j} (\R)^g$. For each $j=1,\dots, \ell,$ let $P_{A^j,X}$ be a matrix whose columns for an orthonormal basis for the kernel of $L_{A^j} (X)$ and set $k^j_{A,X}= \dim \ker L_{A^j} (X)$. Then the linear system \eqref{eq:ArvLinSystem} has a nonzero solution if and only if
		\beq
		\label{eq:GenArvLinSys}
		\Lambda_{A^j} (\beta^T) P_{A^j,X} = 0_{d_j \times k^j_{A,X}} \quad \quad \mathrm{for} \quad \quad j=1, \dots, \ell
		\eeq
	 has a nonzero solution where $\beta \in M_{n \times 1} (\R)^g$. 
		
		For each fixed $j$, the equation $\Lambda_{A^j} (\beta^T) P_{A^j,X} = 0_{d_j \times k^j_{A,X}}$ is a system of $d_j k^j_{A,X}$ linear equations. Therefore, equation \eqref{eq:GenArvLinSys} is a system of $\sum_{j=1}^\ell d_j k^j_{A,X}$ linear equations in $ng$ unknowns. It follows that if 
		\[
		\sum_{j=1}^\ell d_j k^j_{A,X} < ng,
		\]
		then there is a nonzero $\beta$ which is a solution to equation \eqref{eq:GenArvLinSys} from which the result follows. 
	\end{proof}
	
	We now give the solvability count for Euclidean extreme points.
	
	\begin{thm}
		\label{thm:BeanCountEuc}
		Let $A \in \smdrg$ and let $\cD_A \subset \smrg$ be a free spectrahedron.
		\begin{enumerate}
			\item \label{it:BeanCountEucIrred} If $X \in \euc \cD_A (n)$, then $\frac{g(n+1)}{2} \leq d\kAX$. 
			
			\item \label{it:BeanCountEucPolytope} If $\cD_A$ is a free polytope and $X \in \euc \cD_A$, then $\frac{g(n+1)}{2}  \leq \kAX$.
			
			\item \label{it:BeanCountGeneral} Suppose $A$ is unitarily equivalent to the tuple $\oplus_{j=1}^\ell A^j$ where $A^j \in SM_{d_j} (\R)^g$ for each $j$. For $j=1, \dots, \ell$, set $k^j_{A,X}= \dim \ker L_{A^j} (X)$. If $X \in \euc \cD_A$, then
			\beq
			\label{eq:EucKerCount}
			 \frac{g(n+1)}{2} \leq \sum_{j=1}^\ell d_j k^j_{A,X}.
			\eeq
		\end{enumerate}
	\end{thm}
	
	\begin{proof}
		
		The proof of Theorem \ref{thm:BeanCountEuc} is nearly identical to that of Theorem \ref{thm:BeanCount}. The only difference is that one instead considers the linear system
		\beq
		\label{eq:GenEucLinSys}
		\Lambda_{A^j} (\beta^T) P_{A^j,X} = 0_{n d_j \times k^j_{A,X}} \quad \quad \mathrm{for} \quad \quad j=1, \dots, \ell
		\eeq
		where $\beta^T=\beta\in \smnrg$. This is a linear system of $\sum_{j=1}^\ell nd_j k^j_{A,X}$ equations in  $\frac{gn(n+1)}{2}$ unknowns. 
	\end{proof}
	
	\begin{rem}
	\label{rem:kerSize}

	Proposition \ref{prop:ExtremeLinSys} shows that determining if an element of a free spectrahedron is an Arveson or Euclidean extreme point is equivalent to solving homogeneous linear systems. Subsequently, Theorems \ref{thm:BeanCount}  and \ref{thm:BeanCountEuc} use counts on the number of equations vs. unknowns in the relevant linear systems. 
	
	Under generic conditions, if a homogeneous linear system has at least as many equations as unknowns, then the system has no nontrivial solutions. Thus, one may naively expect that if an element $X$ of a free spectrahedron $\cD_A$ is not an Arveson extreme point of $\cD_A$, then $ng > d\kAX$. As we see later in our experiments, while most non-Arveson extreme points that we find satisfy $ng > d\kAX$, a surprising amount do not ($82$ out of $806$).
	\end{rem}

\ssec{Representation in terms of  free extreme points}
\label{sec:NaturalFreeExt}

The following theorem gives a Caratheodory like bound on the number of free extreme points needed to recover an element of a bounded free spectrahedron through matrix convex combinations.

\begin{thm} \cite[Theorem 1.3]{EH19}
\label{thm:FreeCaratheodory}
Let $A \in \smdrg$ such that $\cD_A$ is a bounded free spectrahedron. Given a tuple $X \in \cD_A(n)$, let 
$\mu$ be the dimension of the space of solutions of
\beq
\label{eq:defmu}
\ker L_A (X) \subseteq \ker \Lambda_A (\beta^T)
\eeq
where $\beta \in M_{n \times 1} (\R)^g$. Then there exists an integer $k \leq \mu \leq ng$ such that $X$ is a matrix convex combination of free extreme points of $\cD_A$ whose sum of sizes is equal to $n+k \leq n(g+1)$. 

That is, there exist a collection of free extreme points $\{Z^j\}_{j=1}^m$ of $\cD_A$ such that each $Z^j$ is a tuple of $n_j \times n_j$ matrices and a collection of contractions $\{V_j\}_{j=1}^m$ such that
\[
X= \sum_{j=1}^m V_j^T Z^j V_j \mathand \sum_{j=1}^m V_j^T V_j =I_n \mathand \sum_{j=1}^m n_j = n+k \leq n(g+1). 
\]
In particular one has $m \leq n(g+1)$. 
\end{thm}

A recent work of Hartz and Lupini establishes a different Caratheodory theorem for general closed and bounded matrix convex sets, see \cite[Theorem 1.8]{HL+}. This result shows that if $K \subset \smrg$ is the matrix convex hull of some set $E$ and $X \in K(n)$, then $X$ can be expressed as a matrix convex combination of elements of $E$ with sum of sizes at most $n^2 (g+1)$. For bounded free spectrahedra, the bound obtained from Theorem \ref{thm:FreeCaratheodory} is notably lower; however, \cite[Theorem 1.8]{HL+} has the advantage that it may be applied to any matrix convex set which is closed and bounded.

\ssec{Computing representations of tuples as sums of extreme points} 

A property of free extreme points that is satisfying from a computational perspective is that the representations of an element of a free spectrahedron as a matrix convex combination of elements of the free spectrahedron can be computed (with probability 1) using semidefinite programming. Furthermore, these computed representations satisfy the free Caratheodory bound given in Theorem \ref{thm:FreeCaratheodory}. This we now describe. 

Computation of such a representation is accomplished by computing a sequence of maximal 1-dilations of a given element of a free spectrahedron. Given a free spectrahedron $\cD_A$ and a tuple $X \in \cD_A (n)$, say the dilation 
\[
Y=\begin{pmatrix}
X & c\hbeta \\
c\hbeta^T & \hgamma
\end{pmatrix} \in \cD_A (n+1)
\]
is a \df{maximal $1$-dilation} of $X$ if $\hbeta \in M_{n \times 1} (\R^g)$ is nonzero tuple satisfying $\ker L_A (X) \subset \ker \Lambda_A (\hat{\beta}^T)$ and the real number $c$ is a solution to the maximization problem
\[
\begin{array}{rllcl}  c:=&\underset{\alpha \in \R, \gamma \in \R^g}{\mathrm{Maximizer}} \ \ \ \ \alpha \\
\mathrm{s.t.} & \begin{pmatrix}
X & \alpha \hbeta \\
\alpha \hbeta^T & \gamma
\end{pmatrix} \in \cD_A(n+1) \\
\end{array}
\]
and $\hgamma$ is a classical extreme point of the classical spectrahedron 
\[
\frak{D}_{A,X,c\hbeta}:=\left\{ \gamma \in \R^g \ \bigg|\  L_A \begin{pmatrix}
X & c \hbeta \\
c \hbeta^T & \gamma
\end{pmatrix} \succeq 0 \right\}.
\]

An Arveson dilation of a given element of a bounded free spectrahedron may be computed in the following manner. 

\begin{algo}
\label{algo:ArvDiAlg}
Let $A \in \smdrg$ such that $\cD_A$ is a bounded real free spectrahedron. Given a tuple $X \in \cD_A (n)$, set $Y^0=X$. For integers $j=0,1,2\dots$ and while $Y^j$ is not an Arveson extreme point of $\cD_A$, define 
\[
Y^{j+1}:= \begin{pmatrix} 
Y^j & c_j \hbeta^j \\
c_j (\hbeta^j)^T & \hgamma^j 
\end{pmatrix}
\]
where $\hbeta^j$ is a nonzero solution to 
\[
\ker L_A (Y^j) \subset \ker \Lambda_A (\beta^T) \quad \quad \quad \beta \in M_{n \times 1} (\R^g)
\]
and where $c_j$ and $\hgamma^j$ are solutions to the sequence of maximization problems
\[
\begin{array}{rllcl}  c_j:=&\underset{c \in \R, \gamma \in \R^g}{\mathrm{Maximizer}} \ \ \ \ c \\
\mathrm{s.t.} &L_A \begin{pmatrix}

Y^j & c\hbeta^j \\
c(\hbeta^j)^T & \gamma
\end{pmatrix}\succeq 0 ,\\
\\
\mathrm{and } \quad \hgamma^j:=&\underset{\gamma \in \R^g}{\mathrm{Maximizer}} \ \ \ell(\gamma) \\
\mathrm{s.t.}  &L_A \begin{pmatrix}

Y^j & c_j \hbeta^j \\
c_j (\hbeta^j)^T & \gamma
\end{pmatrix}\succeq 0.
\end{array}
\]
Here $\ell$ is any linear functional which maps $\R^g$ to $\R$.
\end{algo}

\begin{prop}
\label{prop:ArvDiAlg}
Let $\cD_A$ be a bounded free spectrahedron $\cD_A$ and let $X \in \cD_A (n)$.
Then with probability $1$ Algorithm \ref{algo:ArvDiAlg} terminates after $k \leq \mu \leq ng$ steps and the result $Y^k  \in SM_{n+k} (\R)^g$ is an Arveson extreme point of $\cD_A$ which is a dilation of $X$.\footnote{ If the algorithm is designed so that each $\hgamma^j$ is a classical extreme point of the relevant spectrahedron, then this method always succeeds in at most $\mu$ steps.} Here $\mu$ is the dimension of the space of solutions to equation \eqref{eq:defmu}.

In addition, if is $U$ be a unitary such that $U^* Y^k U = \oplus_{i=1}^\ell Z^i$ where each $Z^i$ is irreducible, then each $Z^i$ is a free extreme point of $\cD_A$ and $X$ may be expressed as a matrix convex combination of the $Z^i$ and this matrix convex combination satisfies the free Caratheodory bound of Theorem \ref{thm:FreeCaratheodory}.

\end{prop}
\begin{proof}
Let $Y^0, Y^1 , Y^2 \dots $ be as in Algorithm \ref{algo:ArvDiAlg}. Then for each $j$ the tuple $\hgamma^j$ is a classical extreme point of the classical spectrahedron $\frak{D}_{A,Y^j, c_j \hbeta^j}$  with probability $1$. Therefore for each $j$ the tuple $Y^{j+1}$ is a maximal $1$-dilation of $Y^j$ with probability $1$. 

Assuming each $Y^{j+1}$ is in fact a maximal $1$-dilation of $Y^j$, \cite[Theorem 2.4]{EH19} shows that there is a $k \leq \mu$ such that $Y^k$ is an Arveson extreme point of $\cD_A$ from which the result follows.  

The fact that the $Z^i$ in the second statement of the Proposition are all free extreme is immediate from Theorem \ref{thm:ArvImpliesEuc}. Rewritting the compression of $Y^k$ to $X$ as a matrix convex combination of the $Z^i$ is a routine argument for matrix convex combinations and is omitted. The sum of sizes of the $Z^i$ is equal to $n+k$, the size of $Y^k$, thus the sum of sizes of the $Z^i$ is less than or equal to $n(g+1)$.
\end{proof}

Since there is freedom at each step in the choices of $\beta^j$ and of linear functional used produce each $\hgamma^j$, repeated application of Algorithm \ref{algo:ArvDiAlg} does not produce a unique Arveson dilation for the starting tuple. However, this method will produce some Arveson dilation so long as each $\hgamma^j$ is a classical extreme point  $\frak{D}_{A,Y^j, c_j \hbeta^j}$.

\sssec{Comparison to \cite{EH19}}
\cite[Proposition 2.5]{EH19} offers a method to compute Arveson dilations of a given element of a bounded free spectrahedron. However, this method involves local optimization of the Euclidean norm of an element of a spectrahedron. In contrast, Algorithm \ref{algo:ArvDiAlg} only involves solving linear systems and semidefinite programming.

The definition of a maximal 1-dilation given here differs slightly from the original definition in \cite{EH19} in that $\gamma$ was originally required to be an elemental with maximal Euclidean norm in the classical spectrahedron $\frak{D}_{A,X,c\hbeta}$. However, the definition given here may be used in the main results of \cite{EH19} without modification of the resulting statements or proofs. The primary advantage of the present definition is that a classical extreme point of $\frak{D}_{A,X,c\hbeta}$ may be computed (with probability 1) by optimizing a linear functional over $\frak{D}_{A,X,c\hbeta}$.

\section{Experiment Methodology}

\label{sec:ExperimentMethodology}

The topic of this paper is optimizing a linear functional $\ell $ over a free spectrahedron and the expected nature of its optimizer $\hX^\ell$.
More precisely,
given a bounded free spectrahedron $\cD_A$, a positive integer $n$, and a
linear functional $\ell$ on $\smnrg$, we study properties of
	\beq
	\label{eq:CandidateSDP}
	\begin{array}{rllcl}  \hX^\ell :
	=&\underset{X \in \cD_A(n)}{\mathrm{Minimizer}}
	 \ \ \ \ \ell (X).
	\end{array}
	\eeq
 If $\ell$ is chosen at random from a uniform distribution, then with probability $1$ the minimizer $\hX^\ell$ will be a Euclidean extreme point of $\cD_A$. However, we shall see experimentally that $\hX^\ell$ is also nearly certain to be free extreme. In addition to checking if a minimizer is free or Euclidean extreme, we also examined the dimensions of $\ker L_A (\hat{X}^\ell)$ which occur, see Sections \ref{sec:kerSize} and \ref{sec:KerDimDistri}, as well as  the probability with which a minimizer is irreducible, see Section \ref{sec:irredExp}.

	\ssec{Generating $A$ and $\ell$}
	\label{sec:GenerationOfAell}
	 We now describe how we randomly generate free spectrahedra $\cD_A$ and linear functionals $\ell$. We first discuss the generation of linear functionals. 
	
	\sssec{Random Linear Functionals}
	In the semidefinite program \ref{eq:CandidateSDP}, the linear functional $\ell: \smnrg \to \R$ is randomly generated in one of two methods. One method is to set 
	\beq
	\label{eqn:RationalFunctional}
	\ell(X):=\sum_{k=1}^g \sum_{i \geq j}^n \alpha_{i,j}^k X_{i,j}^k
	\eeq
	where $X_{i,j}^k$ is the $i,j$ entry of $X^k$ and $\alpha_{i,j}^k$ is a random coefficient.  Linear functionals of this form are called \df{random coefficient (\RC) linear functionals}.
	The second method is to define 
	\beq
	\label{eqn:TraceFunctional}
	\ell(X):=\mathrm{tr} \left(V^T V \Lambda_A (X) \right)
	\eeq
	where $V$ is a randomly generated upper triangular $dn \times dn$ matrix. A linear functional of this form is called a \df{random positive weight trace (RPT) linear functional}. 
	
	In these two methods, the coefficients $a_{i,j}^k$ and the nonzero entries of $V$ were randomly generated integers between which where uniformly drawn from an interval $[-b,b]$ where the bound $b$ ranged from $20$ to $200,000$ over the course of the experiments. As a normalization factor, the coefficients $a^k_{i,j}$ and the entries of the matrix $V$ were divided by a constant $d$ which ranged from $10$ to $100,000$, depending on the choice of $b$, so that the final $a^k_{i,j}$ and entries of $V$ were in the range of -2 to 2.
	
	\sssec{Other distributions tested for generating linear functionals}
	In addition to the distribution explained above, in a small number of experiments with \RC \ linear functionals we used two additional distributions for the $\alpha_{i,j}^k$. 	\ben
	\item
	We allowed the $\alpha_{i,j}^k$ to be real mean 0 standard deviation 1 random Gaussian distributed numbers.
	\item
    Additionally, we allowed the $\alpha_{i,j}^k$ to be uniformly distributed real numbers between $-2$ and $2$. 
	
	\een
	
The results obtained using these methods differ little from what is described in this article, so we did not continue experiments with these alternative methods and we omit the details of these results.

 \sssec{Numerical verification of irreducibility}
 
 \label{sec:NumVerIrred}
	 
	 To determine whether a tuple of symmetric matrices $B = (B_1, \dots, B_g) \in \smxrg[m]$ is irreducible, we compute the dimension of its symmetric commutant. Here, the \df{symmetric commutant} of a matrix tuple is the the space of all symmetric matrices which commute with $B_1,  \dots, B_{g-1}$ and $B_g$. The dimension of the symmetric commutant of $B$ is determined by finding the singular values of the linear map $\phi_B:\smxr[m] \to \smxrg[m]$ defined by
	\[
	\phi_B (Z):= ZB-BZ
	\]
	for any $Z \in \smxr[m]$. Section \ref{sec:ZeroDetermine} discusses the methodology used to decide which singular values (or eigenvalues) of a numerical linear map are treated as zero.  
	
Note that irreducibility over $\R$ is not equivalent to other common definitions of irreducibility which are instead equivalent to irreducibility over $\C$. Namely, a matrix tuple $X \in \smnrg$ which is irreducible over $\R$ may fail to generate the algebra $M_n (\R)$. Additionally it is possible for a matrix which is not a multiple of the identity to commute with a tuple of symmetric matrices which is irreducible over $\R$. However, as shown in \cite{EH19}, a tuple $X \in \smnrg$ is irreducible over $\R$ if and only if the only symmetric matrices which commute with $X$ are multiples of the identity.

	\subsubsection{ Generation of free spectrahedra} 
	
	Say a free spectrahedron $\cD_A$ is \df{irreducible} if there is an irreducible tuple $B$ such that $\cD_A=\cD_B$. In our experiments we primarily focus on bounded irreducible free spectrahedra.
		
 Irreducible matrix tuples that define bounded irreducible free spectrahedra are generated in the following manner. We generate first a matrix $\tilde{A} \in M_d(\bbR)^g$, where each entry is a random integer (uniformly distributed) between $-25$ and $25$. Then, we set 
	$$A=\frac{1}{10}(\tilde{A}+\tilde{A}^T).$$ 
We then verify that $A$ is irreducible by determining the dimension of the symmetric commutant of $A$. 

To verify that $\cD_A$ is bounded it is sufficient to show that $\cD_A (1)$ is bounded \cite{HKM13}. One may verify that level one is bounded by checking if there is a cube which contains $\cD_A (1)$.

	\ssec{Properties of optimizing $\hat{X}$.}
	After generating a candidate $\hat{X}$ which is an minimizer of equation \eqref{eq:CandidateSDP}, the tuple is tested for irreducibility. If the tuple $\hat{X}$ is reducible, then it is discarded from further analysis. This is because if $\hat{X} \in \cD_A(n)$ is irreducible, then $\hat{X}$ is a direct sum of two smaller tuples, say of size $n_1$ and $n_2$. Therefore, the properties of $\hat{X}$ may be inferred from properties of tuples of size $n_1$ and $n_2$ and are not necessarily inherent to tuples of size $n$.
	
	If $\hat{X}$ is irreducible, then we test if $\hat{X}$ is a free or Euclidean extreme point of $\cD_A$. 
	We check if $\hat{X}$ is free extreme by calculating the dimension of the kernel of the linear map 
	\beq
	\label{eq:ArvEucMap}
	\psi_{A,\hat{X}}(\cdot):=\Lambda_A (\cdot) P_{A,\hat{X}}:M_{1 \times n} (\R)^g \to M_{d \times k_{A,\hat{X}}}(\R).
	\eeq
	In this definition, $P_{A,\hat{X}}$ is a $dn \times k_{A,\hat{X}}$ matrix whose columns form an orthogonal basis for the kernel of $L_A (\hat{X})$ and $k_{A,\hat{X}}$ is the dimension of the kernel of $L_A (\hat{X})$. As previously discussed in Section \ref{sec:ExtLinSys}, an irreducible $\hat{X}$ is free extreme if and only if $\dim \ker \psi_{A,\hat{X}}=0$.
	
	If $\hat{X}$ is free extreme, then $\hat{X}$ is Euclidean, 
	see \cite[Theorem 1.1]{EHKM18}. 
	If $\hat{X}$ is not free extreme, then we determine if $\hat{X}$ is Euclidean extreme by finding the dimension of the kernel of
	 $\psi_{A,\hat{X}}$ as a map from 
	 $\smnrg \to M_{dn \times k_{A,\hat{X}}}(\R)$. In either setting, the dimension of the kernel of $\psi_{A,\hat{X}}$ is estimated by computing the singular values of $\psi_{A,\hat{X}}$.
	 
 In addition, to determining if $\hat{X}$ is Arveson or Euclidean extreme. The dimension of the kernel of $L_A(\hat{X})$ is recorded.
	
	Based on (soon to be presented) experimental evidence, 
	irreducible tuples which are Euclidean extreme 
	but are not free extreme, i.e. non-Arveson extreme points, rarely occur as optimizers.  As experiments run, free spectrahedron and linear functional pairs which generate non-Arveson extreme points are automatically stored in separate files so these tuples can be examined in greater detail.

	\subsection{The ``what is zero" decisions} 
	\label{sec:ZeroDetermine}
	
	Let $\sigma \in \R^n$ be a list of numerically computed eigenvalues or singular values of a nonzero linear map where the entries of $\sigma$ are ordered so $|\sigma_j| \geq |\sigma_{j+1}|$ for all $j=1, \dots, n-1$. We determine which entries of $\sigma$ to treat as zero by setting tolerances $\epsilon_1$ and $\epsilon_2$ and searching for the smallest index $1<j\leq n$ so that 
	\[
	|\sigma_j|<\epsilon_1 \quad \quad \mathrm{and} \quad \quad \frac{|\sigma_j|}{|\sigma_{j-1}|}< \epsilon_2.
	\]
	In words, we search for the smallest index $j$ such that $\sigma_j$ is sufficiently small and such that the relative gap between $\sigma_j$ and $\sigma_{j-1}$ is sufficiently large. If such an index $j_0$ is found, then we consider $\sigma_j$ to be equal to zero for all indices $j$ such that $j_0 \leq j \leq n$. 
	
	If no such $j_0$ is found, but there is an index $j'_0$ such that
	\[
	|\sigma_j|<\epsilon_1 *10^\alpha \quad \quad \mathrm{and} \quad \quad \frac{|\sigma_j|}{|\sigma_{j-1}|}< \epsilon_2  *10^\alpha.
	\]
	then we report that the numeric of the problem are ill conditioned and that the presence or absence of zero eigenvalues or singular values cannot reliably be determined. Here $\alpha = 2$ when considering eigenvalues, and $\alpha=1$ when considering singular values. If no such $j_0$ or $j'_0$ is found, then the we determine that all eigenvalues or singular values are nonzero.
	
	When determining the dimension of the kernel of $L_A (\hat{X})$ we set $\epsilon_1=10^{-6}$ and $\epsilon_2=10^{-5}$. Meanwhile, when  classifying $\hat{X}$ as free or Euclidean extreme we set $\epsilon_1=10^{-3}$ and $\epsilon_2=10^{-2.5}$. Finally, when determining the dimension commutant of $\hat{X}$ we set $\epsilon_1=10^{-4.5}$ and $\epsilon_2=10^{-4}$. We again remark that the computations for classifying $\hat{X}$ as Arveson or Euclidean extreme and determining the dimension of the commutant of $\hat{X}$ all use singular values, while finding the dimension of the kernel of $L_A (\hat{X})$ uses eigenvalues. We performed a series of runs with different tolerances to confirm that the choices stated above were effective while other choice lead to various difficulties. 
	
	Table \ref{table:TolerancesTable} contains the values of $\epsilon_1$ and $\epsilon_2$ used in each test, as well as whether eigenvalues or singular values were used as a diagnostic. When an experiment did not meet all numerical tolerances, it was discarded. The discard rate with two exceptions was roughly 4\%, for a few details see Section \ref{sec:nonArvRare}.
	
	\begin{table}[h]
	\caption{Table of tolerances}
		\begin{tabular}{|l|l|l|l|l|}
			\hline
			Computation & $\epsilon_1$ & $\epsilon_2$ & Diagnostic\\ \hline
			Kernel dimension & $10^{-6}$ & $10^{-5}$ & Eigenvalues  \\ \hline
			Free extreme & $10^{-3}$ & $10^{-2.5}$ & Singular values  \\ \hline
			Euclidean extreme & $10^{-3}$ & $10^{-2.5}$ & Singular values  \\ \hline
			Irreducibility & $10^{-4.5}$ & $10^{-4}$ & Singular values   \\ \hline
		\end{tabular}
	
	\label{table:TolerancesTable}
	\end{table}

\subsubsection{Numerical Issues}	
\label{sec:NumErr}

We individually examined several of the minimizers $\hX^\ell$ that our experiments rejected because an error tolerance was not met. In all cases we checked closely, it was the kernel dimension tolerance which was violated. There was not a sharp enough drop in the plot of eigenvalues to be sure where the null space started. Moreover, we found that the putative null space was relatively large in comparison to other null spaces observed at the same level of that spectrahedron.

 Based on the idea that a tuple $X$ is more likely to be free extreme point the larger $\ker L_A (X)$ is, we believe that this exclusion has little effect on our results presented in Section \ref{sec:nonArvRare} and Section \ref{sec:kerSize} where we show that we typically observe free extreme points and that we typically observe large kernels. However,  exclusion of kernels of extreme size may have some impact on results in Section \ref{sec:KerDimDistri} where the distribution of observed kernel sizes is discussed. Additionally, we often see that tuples with very large kernels are more likely to be reducible. As a consequence, this exclusion may impact results in Section \ref{sec:irredExp}, where we discus the proportion of minimizers which are irreducible. 

A second possible source of numerical error is that the determination of the dimension of the commutant of a tuple can be sensitive to the thresholds set for zero decisions. In rare cases a (nearly) reducible tuple may have been determined to be irreducible. For example, whether or not one calls the tuple
\[
\left(
\begin{pmatrix}
1 & 0 \\
0 & -1
\end{pmatrix},
\begin{pmatrix}
1 & \epsilon \\
\epsilon & -1
\end{pmatrix}\right)
\]
reducible depends on the threshold at which one treats $\epsilon$ as zero. 
	
	\subsection{Our Experiments}
	
	We ran large numbers of experiments of two general types:

	\begin{enumerate}
    	\item ($A$, $\ell$ pairs) \index{$A$, $\ell$ pairs}: Fix $g,d,n \in \mathbb{N}$. Pairs consisting of a bounded free spectrahedron $\cD_A$ and a (\RC \  or \RPT) linear functional $\ell$ were randomly generated. For each $(\cD_A,\ell)$ pair, the minimizer $\hat{X}^\ell \in \smnrg$ was computed.
    	
	    \item (Fixed random $A$, many $\ell$) 
	   \index{Fixed random $A$, many $\ell$}: A bounded free spectrahedron $\cD_A$ was generated. Then for several choices of $n$ (typically $1 \leq n \leq 8$) large numbers of \RC \  and \RPT \  linear functionals were generated and minimized over $\cD_A (n)$ to generate extreme points of $\cD_A$ at level $n$.
	\end{enumerate}
	
	Typically $g=2, \dots, 6$,\ and $ g \leq d \leq 7$. Based on a small number of examples, the restriction $g \leq d$ appears to be important for our observations and we expect that different behaviour may occur when $g>d$. Typically $n$ is less than or equal to $8$ but we occasionally allow $n$ up to $14$.
  
  The experiments either
  \ben
  \item
  fix $g, d, n \in \mathbb{N}$ and randomly generate thousands of
  different $(\cD_A,\ell)$ pairs where $A$ is irreducible and $\cD_A$ is bounded, OR
  \item
  for each of 60 selected (randomly generated) $\cD_A$, fix $\cD_A$ and generate thousands of random linear functionals $\ell$ for each level $n=1, \dots, 8$. \\
  \een

In either case, for each linear functional $\ell$ and free spectrahedron $\cD_A$ pair, the minimizer  $X^\ell$ is computed.
 The total number of cases is approximately 7.3 million.

	In describing our experimental findings, we frequently use the term \df{it is \ncr }. For example, as we soon see it is nearly certain that the minimizer $\hX^\ell$ of a \RC \  or \RPT \ 
linear functional $\ell$ over a random bounded irreducible free spectrahedron $\cD_A$ defined by $A \in \smdrg$ with $g < d$ is \fep \ provided it is irreducible. We use this phrasing rather than the more common ``with high probability", since this usually occurs in papers where these are proven estimates on how high this probability is. We do not have such estimates, so we avoid confusion by using a different terminology. Also in our findings we see that exceptions to the pattern we find are very rare, so strong language is warranted.

	\section{Non-Arveson Extreme Points are Rare}
	\label{sec:nonArvRare}
		This section gives a list of findings related to irreducible spectrahedra, namely spectrahedra $\cD_A$ where $A$ is an irreducible tuple.
	 The irreducible minimizers our experiments find of course are Euclidean extreme points, but, very surprisingly, are nearly certain to be Arveson extreme points. In this section we give more detail.

	\begin{obs}
	\label{obs:irredNonArvLevelUpperBound}
	Fix $d > g > 2$ or $g=d=4 \mathrm{ \ or \ } 5$. Then for a randomly generated $A \in \smdrg$ such that $\cD_A$ is a bounded irreducible free spectrahedron, it is nearly certain that the minimizer $\hX^\ell$ of a RC or RPT linear functional $\ell$ over $\cD_A (n)$ is a free extreme point of $\cD_A$ if $\hX^\ell$ is irreducible. 
	
	Furthermore, for these values of $g,d$, we observe that there is an integer $N_{g,d}$ depending only on $g$ and $d$ such that $\hX^\ell$ is always a free extreme point of $\cD_A (n)$ provided that $\hX^\ell$ is irreducible and $n \geq N_{g,d}$. If $g \geq 5$ then we observe that $N_{g,d}=1$.

	\end{obs}

	\ssec{Tables counting non-Arveson extreme points} 
	\label{sec:FEcount}
	The evidence for Observation \ref{obs:irredNonArvLevelUpperBound} is in the tables which follow. The tables below give a complete list of the irreducible non-Arveson extreme points found in our experiments when $d>g>2$ and when $g=d=4$ and (implicitly) when $g=d=5$. 
	
	For $d>g>2$ or $g=d= 4 \mathrm{\ or \ } 5$, we did a total of 3,912,000 runs using RC linear functionals. In these runs, about 3,405,000 optimizers were irreducible extreme points. Out of these only 315 optimizers were irreducible non-Arveson extreme points, which is approximately  0.00925\%. We also did 2,926,000 
runs using RPT linear functional. In these runs, about 2,541,000 optimizers were irreducible extreme points out of which only 73 optimizers were irreducible non-Arv extreme points, which is approximately 0.00287\%.

	We discarded 4.09 \% of all our runs because one of our error tolerances was not met. This includes two bad outliers. For $g=d=3$ using RC linear functionals, the discard rate is about 10 \%. For $g=d=4$ using RC linear functionals and RPT linear functionals, the discard rates are 19 \% and 11 \% respectively. In all other cases, the discard rates were no greater than 7.6 \%.

	\begin{table}[H]
		\caption{
		$A$, $\ell$ pairs (RC). Total experiments: 1020000.  \newline
		 Irred non-Arveson/Total irred extreme : $ 51/ 916447 \approx 0.00556 \%$ \newline
		 All extreme: 966576. Num. reducibles: 50129. Num. errors: 53424.
		  \newline
		 The above includes 100000 points at level 1, 99390 of which are extreme.\protect\textsuperscript{\ref{note:level1nonextreme}}
		} 
	\begin{adjustbox}{width=\textwidth}
		\begin{tabular}{|c|c|c|c|c|c||c|c|c|}
			\hline
			\multirow{2}{*}{$g$} & \multicolumn{5}{|c||}{All irred non-Arveson extreme points} & \multicolumn{3}{c|}{All runs} \\ \cline{2-9}
			& $d$ & $n$ & $\dim \ker$ ($k$) & $g n-d k$ & \# non-free & \# runs per d,n & range of $d$ & range of $n$\\ \hline
			\multirow{4}{*}{3} & 4 & 3 & 2 & 1 & 27 & \multirow{4}{*}{10000} & \multirow{4}{*}{4-7} & \multirow{4}{*}{\begin{tabular}[c]{@{}c@{}} 1-8 for d=4 \\ 1-13 for d=5-7 \end{tabular}} \\ \cline{2-6}
			& 4 & 6 & 4 & 2 & 1 &  &  &  \\ \cline{2-6}
			& 4 & 7 & 5 & 1 & 2 &  &  &  \\ \cline{2-6}
			& 5 & 2 & 1 & 1 & 15 & &  & \\ \hline
			\multirow{6}{*}{4} & 4 & 2 & 2 & 0 & 1 & \multirow{6}{*}{10000} & \multirow{6}{*}{4-6} & \multirow{6}{*}{1-8} \\ \cline{2-6}
			& 4 & 6 & 6 & 0 & 1 & &  & \\ \cline{2-6}
			& 4 & 7 & 7 & 0 & 1 & &  & \\ \cline{2-6}
			& 4 & 8 & 8 & 0 & 1 & &  & \\ \cline{2-6}
			& 5 & 3 & 2 & 2 & 1 & &  & \\ \cline{2-6}
			& 6 & 2 & 2 & -4 & 1 & &  & \\ \hline
			5 & - & - & - & - & 0 & 10000 & 5-7 & 1-8 \\ \hline
			6 & - & - & - & - & 0 & 10000 & 7 & 2-8 \\
			\hline
		\end{tabular}
	\end{adjustbox}
		\label{table:YiNFEPRC}
	\end{table} 
	
	\stepcounter{footnote}
	\footnotetext{\label{note:level1nonextreme} The points at level $1$ which were not extreme points were numerically ill-conditioned.}

	\begin{table}[H]
		\caption{Fixed $A$, many $\ell$ (RC). Total experiments: 2892000.  \newline
		 Irred Non-Arveson/Total irred extreme: $  264 /  2488251 \approx 0.01061\%$ \newline
		 All extreme:  2728443. Num. reducibles: 240192. Num. errors:  163557.
		  \newline
		 The above includes 249000 points at level 1, 244110 of which are extreme.\protect\textsuperscript{\ref{note:level1nonextreme}}
	}
\begin{adjustbox}{width=\textwidth}
		\begin{tabular}{|c|c|c|c|c|c||c|c|c|}
			\hline
			\multirow{2}{*}{$g$} & \multicolumn{5}{|c||}{All irred non-Arveson extreme points} & \multicolumn{3}{c|}{All runs} \\ \cline{2-9}
			& $d$ & $n$ & $\dim \ker$ ($k$) & $g n-d k$ & \# non-free & \# runs per d,n & range of $d$ & range of $n$\\ \hline
			\multirow{3}{*}{3} & 4 & 3 & 2 & 1 & 94 & \multirow{3}{*}{\begin{tabular}[c]{@{}c@{}c@{}} 77000 for d=4-6, n=2-4 \\ 27000 for d=4-6, n=1,5-8 \\ 15000 for d=7 \end{tabular}} & \multirow{3}{*}{4-7} & \multirow{3}{*}{1-8} \\ \cline{2-6}
			& 5 & 2 & 1 & 1 & 160 & & & \\ \cline{2-6}
			& 5 & 4 & 2 & 2 & 1 & & & \\ \hline
			\multirow{5}{*}{4} & 4 & 2 & 2 & 0 & 1 & \multirow{5}{*}{\begin{tabular}[c]{@{}c@{}} 67000 for n=2-4 \\ 17000 for n=1,5-8 \end{tabular}} & \multirow{5}{*}{4-6} & \multirow{5}{*}{1-8} \\ \cline{2-6}
			& 4 & 3 & 3 & 0 & 1 & & & \\ \cline{2-6}
			& 4 & 4 & 4 & 0 & 2 & & & \\ \cline{2-6}
			& 5 & 3 & 2 & 2 & 2 & & & \\ \cline{2-6}
			& 5 & 4 & 3 & 1 & 3 & & & \\ \hline
			5 & - & - & - & - & 0 & 34000 & 5-7 & 1-8 \\ \hline
		\end{tabular}
	\end{adjustbox}
		\label{table:JohnNFEPRC}
	\end{table}

		\begin{table}[H]
		\caption{
			$A$, $\ell$ pairs (RPT). Total experiments: 1030000.  \newline
			Irred non-Arveson/Total irred extreme : $ 17/ 876052 \approx 0.00194 \%$  \newline
			All extreme: 998239. Num. reducibles: 122187. Num. errors: 31761.
			\newline
			The above includes 110000 points at level 1, 109973 of which are extreme.\protect\textsuperscript{\ref{note:level1nonextreme}}
		} 
	\begin{adjustbox}{width=\textwidth}
		\begin{tabular}{|c|c|c|c|c|c||c|c|c|}
			\hline
			\multirow{2}{*}{$g$} & \multicolumn{5}{|c||}{All irred non-Arveson extreme points} & \multicolumn{3}{c|}{All runs} \\ \cline{2-9}
			& $d$ & $n$ & $\dim \ker$ ($k$) & $g n-d k$ & \# non-free & \# runs per d,n & range of $d$ & range of $n$\\ \hline
			\multirow{2}{*}{3} & 
			\multirow{2}{*}{5} & 
			\multirow{2}{*}{2} & 
			\multirow{2}{*}{1} & 
			\multirow{2}{*}{1} & 
			\multirow{2}{*}{13} & 
			\multirow{2}{*}{10000} & 
			\multirow{2}{*}{4-7} & \multirow{2}{*}{\begin{tabular}[c]{@{}c@{}} 1-8 for d=4 \\ 1-13 for d=5-7 \end{tabular}} \\ 
			&  &  &  &  &  &  &  & \\ \hline
			\multirow{3}{*}{4} & 4 & 2 & 2 & 0 & 1 & \multirow{3}{*}{10000} & 
			\multirow{3}{*}{4-6} & 
			\multirow{3}{*}{1-8} \\ \cline{2-6}
			& 4 & 5 & 5 & 0 & 2 & &  & \\ \cline{2-6}
			& 5 & 4 & 3 & 1 & 1 & &  & \\ \hline
			5 & - & - & - & - & 0 & 10000 & 5-7 & 1-8 \\ \hline
			6 & - & - & - & - & 0 & 10000 & 7 & 1-8 \\
			\hline
		\end{tabular}
	\end{adjustbox}
		\label{table:YiNFEPRPT}
	\end{table} 
	
	\begin{table}[H]
		
		\caption{ Fixed $A$, many $\ell$ (RPT). Total experiments: 1896000.  \newline
		 Irred Non-Arveson/Total irred extreme: $56 /  1665174 \approx  0.00336\%$ \newline
		 All extreme:  1862608. Num. reducibles: 197434. Num. errors:  33392. 
		  \newline
		 The above includes 237000 points at level 1, 236939 of which are extreme.\protect\textsuperscript{\ref{note:level1nonextreme}}
		 } 
	 \begin{adjustbox}{width=\textwidth}
		\begin{tabular}{|c|c|c|c|c|c||c|c|c|}
			\hline
			\multirow{2}{*}{$g$} & \multicolumn{5}{|c||}{All irred non-Arveson extreme points} & \multicolumn{3}{c|}{All runs} \\ \cline{2-9}
			& $d$ & $n$ & $\dim \ker$ ($k$) & $g n-d k$ & \# non-free & \# runs per d,n & range of $d$ & range of $n$\\ \hline
			\multirow{3}{*}{3} & 4 & 3 & 2 & 1 & 26 & \multirow{3}{*}{\begin{tabular}[c]{@{}c@{}} 25000 for d=4-6 \\ 15000 for d=7 \end{tabular}} & \multirow{3}{*}{4-7} & \multirow{3}{*}{1-8} \\ \cline{2-6}
			& 4 & 4 & 3 & 0 & 1 & & & \\ \cline{2-6}
			& 5 & 2 & 1 & 1 & 22 & & & \\ \hline
			\multirow{4}{*}{4} & 4 & 2 & 2 & 0 & 1 & \multirow{4}{*}{15000} & \multirow{4}{*}{4-6} & \multirow{4}{*}{1-8} \\ \cline{2-6}
			& 4 & 3 & 3 & 0 & 1 & & & \\ \cline{2-6}
			& 4 & 4 & 4 & 0 & 2 & & & \\ \cline{2-6}
			& 5 & 4 & 3 & 1 & 3 & & & \\ \hline
			5 & - & - & - & - & 0 & 34000 & 5-7 & 1-8 \\ \hline
		\end{tabular}
	\end{adjustbox}
		\label{table:JohnNFEPRPT}
	\end{table}

 Note that in Table \ref{table:JohnNFEPRC}, for $g = 3$ and $d = 4 -6$  there are runs that were performed on a single spectrahedron only at levels $n =  2-4$. This is why there is a significant difference between the number of runs at levels $2-4$ and the number of runs at other levels for these choices of $g$ and $d$.

\subsection{g=2}

	Note regarding Table \ref{table:YiNFEPRC}: The non-Arveson points for $g=2$ are listed in a separate table, i.e. Table \ref{table:YiNFEPRCg2}. For $g=2$ and $d=3$ we continue to see irreducible non-Arveson extreme points $\hX^\ell$  even for large $n$; however, these non-Arveson extreme points become increasingly rare as $n$ increases.
	
	 \begin{table}[H]
		\caption{$A$, $\ell$ pairs (RC) $g=2$. Total experiments: 401000.  \newline
		 Irred Non-Arveson/Total irred extreme: $ 263/ 254523 \approx 0.103 \%$ \newline
		 All extreme: 380259. Num. reducibles: 125736. Num. errors: 20741. 
		  \newline
		 The above includes 11000 points at level 1, 10980 of which are extreme.\protect\textsuperscript{\ref{note:level1nonextreme}}
		}
	\begin{adjustbox}{width=\textwidth}
		\begin{tabular}{|c|c|c|c|c|c||c|c|c|}
			\hline
			\multirow{2}{*}{$g$} & \multicolumn{5}{|c||}{All irred non-Arveson extreme points} & \multicolumn{3}{c|}{All runs} \\ \cline{2-9}
			& $d$ & $n$ & $\dim \ker$ ($k$) & $g n-d k$ & \# non-free & \# runs per d,n & range of $d$ & range of $n$\\ \hline
			\multirow{19}{*}{2} & 3 & 2 & 1 & 1 & 181 & \multirow{19}{*}{\begin{tabular}[c]{@{}r@{}r@{}} 1000 for g=2,d=3,n=1 \\ 0 for g=2,d=4,n=1 \\ 10000 for the rest of d,n \end{tabular}}  & 
			\multirow{19}{*}{3-5} & 
			\multirow{19}{*}{1-14}\\  \cline{2-6}
			& 3 & 3 & 2 & 0 & 20 & &  &  \\ \cline{2-6}
			& 3 & 4 & 2 & 2 & 2 & &  &  \\ \cline{2-6}
			& 3 & 4 & 3 & -1 & 2 & &  &  \\ \cline{2-6}
			& 3 & 5 & 3 & 1 & 28 & &  &  \\ \cline{2-6}
			& 3 & 5 & 4 & -2 & 6 & &  &  \\ \cline{2-6}
			& 3 & 6 & 4 & 0 & 3 & &  &  \\ \cline{2-6}
			& 3 & 6 & 5 & -3 & 5 & &  &  \\ \cline{2-6}
			& 3 & 7 & 5 & -1 & 2 & &  &  \\ \cline{2-6}
			& 3 & 8 & 5 & 1 & 5 & &  &  \\ \cline{2-6}
			& 3 & 8 & 6 & -2 & 1 & &  &  \\ \cline{2-6}
			& 3 & 8 & 7 & -5 & 2 & &  &  \\ \cline{2-6}
			& 3 & 9 & 7 & -3 & 1 & &  &  \\ \cline{2-6}
			& 3 & 9 & 8 & -6 & 1 & &  &  \\ \cline{2-6}
			& 3 & 10 & 8 & -4 & 1 & &  &  \\ \cline{2-6}
			& 3 & 14 & 12 & -8 & 1 & &  &  \\ \cline{2-6}
			& 3 & 14 & 13 & -11 & 1 & &  &  \\ \cline{2-6}
			& 4 & 4 & 2 & 0 & 1 & &  &  \\ \hline
		\end{tabular}
	\end{adjustbox}
		\label{table:YiNFEPRCg2}
	
	\end{table} 
	
	\begin{table}[H]
		\caption{$A$, $\ell$ pairs (RPT) $g=2$. Total experiments: 420000.  \newline
			Irred Non-Arveson/Total irred extreme: $ 155 / 169771 \approx 0.0913 \%$  \newline
			All extreme: 416019. Num. reducibles: 246248. Num. errors: 3981.
			\newline
			The above includes 30000 points at level 1, 29999 of which are extreme.\protect\textsuperscript{\ref{note:level1nonextreme}} 
		}
	\begin{adjustbox}{width=\textwidth}
		\begin{tabular}{|c|c|c|c|c|c||c|c|c|}
			\hline
			\multirow{2}{*}{$g$} & \multicolumn{5}{|c||}{All irred non-Arveson extreme points} & \multicolumn{3}{c|}{All runs} \\ \cline{2-9}
			& $d$ & $n$ & $\dim \ker$ ($k$) & $g n-d k$ & \# non-free & \# runs per d,n & range of $d$ & range of $n$\\ \hline
			\multirow{14}{*}{2} & 3 & 2 & 1 & 1 & 131 & \multirow{14}{*}{10000} & 
			\multirow{14}{*}{3-5} & 
			\multirow{14}{*}{1-14}\\  \cline{2-6}
			& 3 & 3 & 2 & 0 & 7 & &  &  \\ \cline{2-6}
			& 3 & 4 & 2 & 2 & 2 & &  &  \\ \cline{2-6}
			& 3 & 4 & 3 & -1 & 2 & &  &  \\ \cline{2-6}
			& 3 & 5 & 4 & -2 & 2 & &  &  \\ \cline{2-6}
			& 3 & 6 & 4 & 0 & 2 & &  &  \\ \cline{2-6}
			& 3 & 6 & 5 & -3 & 1 & &  &  \\ \cline{2-6}
			& 3 & 8 & 5 & 1 & 1 & &  &  \\ \cline{2-6}
			& 3 & 9 & 8 & -6 & 1 & &  &  \\ \cline{2-6}
			& 3 & 11 & 7 & 1 & 1 & &  &  \\ \cline{2-6}
			& 3 & 12 & 10 & -6 & 1 & &  &  \\ \cline{2-6}
			& 3 & 14 & 9 & 1 & 1 & &  &  \\ \cline{2-6}
			& 4 & 2 & 1 & 0 & 2 & &  &  \\ \cline{2-6}
			& 4 & 5 & 2 & 2 & 1 & &  &  \\ \hline
		\end{tabular}
	\end{adjustbox}
		\label{table:YiNFEPRPTg2}
	\end{table} 
	
	\ssec{$g=d=3$}
	The $g=d=3$ case is dramatically different than the other cases in our experiments. As such, this case is excluded from discussion in the upcoming sections. A brief overview of our findings for $g=d=3$ follows.
	
	For $g=d=3$, we generated $10,000$ random linear functionals and optimizers for each level $n=1, \dots, 8$. For $n \ge 5$, more than 10\% of the optimizers are numerically ill-conditioned. This percentage increases to 18\% when $n = 8$. For $n \ge 2$, more than 60\% of the extreme points are reducible. The percentage of reducible extreme points gets as high as 80\% when $n = 8$. For larger $n$,  free extreme points tend to have the same kernel dimension, while the irreducible non-Arveson extreme points have a wide variety of kernel dimensions. 
	
	In the $g=d=3$ case, we again find that free extreme points are significantly more common than irreducible non-Arveson extreme points. However, when compared to other values of $g$ and $d$ where irreducible non-Arveson extreme points are extremely rare, irreducible non-Arveson extreme points are surprisingly common when $g=d=3$. In this case, we find that about $10\%$ percent of the irreducible optimizers found in our experiments are non-Arveson extreme points.

	\section{Dimension of $\ker L(\hX^\ell)$ Conjectures}
	\label{sec:kerSize}
	
		The irreducible minimizers $\hX^\ell$ that our experiments find are nearly certain to have  $L_A(\hX^\ell)$ with  surprisingly large kernels,
 namely,
	 \beq
	 \label{eq:gnVSdk}
		\frac{gn}{d} \leq \dim \ker L_A(\hX^\ell) .
	\eeq
 In light of Theorem \ref{thm:BeanCount}, this
 is very consistent with the finding in the previous section	 
that the $\hX^\ell$ are rarely non-Arveson.
Thus, while \S \ref{sec:nonArvRare}
and \S \ref{sec:kerSize}
are each surprising by themselves, one of them 
is not so surprising if you know the other.
	 
	As usual, the experiments we report on in this section are all performed on  irreducible free spectrahedra.

	\subsection{Upper and lower bound for the $L(\hX^\ell)$ kernel dimensions}
	\label{sec:bd kerdim}
	
	The following gives our observed upper and lower bounds on the dimension of the kernel of $L_A (\hX^\ell)$ for irreducible minimizers $\hX^\ell$. Here $\cD_A$ is an irreducible free spectrahedra.\\
	
	{
	  \begin{obs}
        \label{obs:BoundsOnKerDim}
		Fix $2\leq g < d $ or $g=d= 4 \mathrm{\ or \ } 5$. In our experiments,
		an irreducible minimizer $\hX^\ell$ 
		\ben
		\item
		\label{it:KerNCRLowerBound}
		is nearly certain to satisfy
		$$
		\frac{gn}{d} \leq \dim \ker L_A(\hX^\ell)
		$$
		\item
		\label{it:KerUpperBound}
		 (always) satisfies 
		$\dim \ker L_A(\hX^\ell) \leq  2n$
		\item
		\label{it:NonArvKerTypicalUpperBound} most non-Arveson extreme points satisfy
		$$
		\frac{gn}{d} > \dim \ker L_A(\hX^\ell).
		$$
		\een
	
	\end{obs}

	\evidence
	The evidence for Observation \ref{obs:BoundsOnKerDim} is contained in Tables \ref{table:YiNFEPRC} through \ref{table:YiNFEPRCg2}.
	
	Regarding item \eqref{it:KerNCRLowerBound}: All examples of minimizers $\hX^\ell$ which do not satisfy 
	\[
		\frac{gn}{d} \leq \dim \ker L_A(\hX^\ell)
	\]
	are reported on in the aforementioned tables. Out of over $6$ million total irreducible $\hX^\ell$ computed in our experiments, only about $724$ do not satisfy this bound.
	
	Regarding item \eqref{it:KerUpperBound}: All the minimizers we computed in our experiments satisfy this upper bound. For small values of $n$ this bound is achieved. However, for large values of $n$ this bound is not observed to be sharp. 
	
	Regarding item \eqref{it:NonArvKerTypicalUpperBound}: Out of the non-Arveson extreme points found by our experiments, about $90$ percent satisfy 
	\beq
	\label{eq:BeanSufficiencyGuess}
	\frac{gn}{d} > \dim \ker L_A(\hX^\ell).
	\eeq
	The majority of counter examples occur for free spectrahedra in two variables. If we restrict to $2<g<d$ or $g=d=4$ or $g=d=5$, then approximately $96$ percent of non-Arveson extreme points satisfy inequality \eqref{eq:BeanSufficiencyGuess}.

	Although about $90$ percent of irreducible non-Arveson extreme points satisfy inequality \eqref{eq:BeanSufficiencyGuess}, the number of non-Arveson extreme points we find that do not satisfy the bound is perhaps surprising in contrast to the line of thought mentioned in Remark \ref{rem:kerSize}.
	\qed
	
	In stark contrast to the $g<d$ or $g=d= 4 \mathrm{\ or \ } 5$ cases, less than $1$ percent of non-Arveson irreducible extreme points satisfy inequality \eqref{eq:BeanSufficiencyGuess} for $g=d=3$. 
	
\begin{rem}

We occasionally observe free spectrahedra $\cD_A$ on which the lower bound $\lceil\frac{gn}{d}\rceil$ for the dimension of the kernel of $L_A (\hX^\ell)$ for irreducible minimizers $\hX^\ell$ which are free extreme points of $\cD_A$ is not achieved in our experiments. \qed

\end{rem}

\begin{rem}
			\label{conj:irredLessPoly} If $\cD_A$ is an irreducible free spectrahedron for which the bound in Observation \ref{obs:BoundsOnKerDim} \eqref{it:KerUpperBound} holds and if $\cD_B$ is any bounded free polytope, then Theorem \ref{thm:BeanCount} \eqref{it:BeanCountPolytope} would immediately imply that for all $n \in \bbN$ we have
	\[
	 \max_{X \in \cD_A(n)} \dim \ker L_A(X) \leq   \min_{Y \in \arv \cD_B(n)} \dim \ker L_B(Y).
	 \]
	 The fact that $\cD_B(n)$ has Arveson extreme points for each $n$ is a consequence of \cite[Proposition 6.1]{EHKM18}. \qed
\end{rem}
	
	\subsection{Free vs Non-Free Kernel Dimensions} 
	
	The following  observation  compares the dimensions of kernels of non-Arveson extreme points to Arveson extreme points.
	
	\begin{obs}
		\label{conj:EucMaxLessArvMax}
		Suppose $2 \leq g \leq d$  except for $g=d=2$, and let $\cD_A$ be a bounded irreducible free spectrahedron. Then for any natural number $n$ we observe that one has 
		\begin{enumerate}
		\item 
		\label{it:maxEucVsMaxArv}
		$$\max_{X \in \euc \cD_A(n) \backslash \arv \cD_A(n)} \dim \ker L_A(X)<\max_{X \in \arv \cD_A(n)} \dim \ker L_A(X).$$
		
		\item 
		\label{it:maxEucVsMinArv}
		If one also avoids $g=2, d=3$ and $g=d=3$, then
			$$\max_{X \in \euc \cD_A(n) \backslash \arv \cD_A(n)} \dim \ker L_A(X) \leq \min_{X \in \arv \cD_A(n)} \dim \ker L_A(X).$$
		\end{enumerate}
		
	\end{obs}
	
	\evidence The evidence is purely experimental; no counterexamples were found. 
	
	As to item \ref{it:maxEucVsMaxArv}, for $g=d=3$, the kernel dimension of non-Arveson extreme points may exceed the largest kernel dimension of irreducible Arveson extreme points. However, we do not observe counter examples when allowing reducible Arveson extreme points. With regard to item \ref{it:maxEucVsMinArv}, for $g=2$, $d=3$, $n=14$, we have only three Euclidean non-Arveson extreme points; they are of kernel dimension 9, 12 and 13, and the smallest kernel dimension of Arveson extreme points we observed is 10. For $g=d=3$, the kernel dimension of Euclidean non-Arveson extreme points usually exceed the smallest kernel dimension of Arveson extreme points.

	As there are no counter examples, one could consider making this observation a conjecture.\qed
	
	\begin{rem}
		Although our paper mostly restricts to irreducible minimizers, one can show that if there is a reducible tuple of size $n$ that violates the conjecture, then there is an integer $m < n$ such that the conjecture is violated when restricted to irreducible tuples at level $m$.
		\qed
	
	\end{rem}
	
	\section{How the Kernel Dimensions Are Distributed}
	\label{sec:KerDimDistri}
	
	This section presents patterns our experiment found in the distribution of the dimension of $\ker L_A (\hX^\ell)$ for irreducible minimizers $\hX^\ell$. We restrict our attention to presenting results for RC linear functionals.

	\subsection{Distribution of Kernel Sizes} 
    		
Fix $g,d,n \in \mathbb{N}$ and let
     $\Omega_{g,d,n}$ be the set of pairs $(A,\ell)$ which can arise in our experiments when using \RC \ linear functionals. See Section \ref{sec:GenerationOfAell} for details on which pairs $(A,\ell)$ are admissible in $\Omega_{g,d,n}$. 
     
    Noting that $\Omega_{g,d,n}$ is a finite set, let $\mu_{g,d,n}$ denote the uniform measure on $\Omega_{g,d,n}$. Define $K_{g,d,n}$ to be the random variable on sample space $(\Omega_{g,d,n}, \mu_{g,d,n})$ defined by $K_{g,d,n}(A,\ell) = \dim \ker(\hX^\ell)$.

	\begin{obs}
			 \label{obs:KgdnIsUnimodal} Fix $2 < g < d$ or $g=d= 4 \mathrm{\ or \ } 5$, and let $n \in \bbN$. Then, the probability density function (PDF) \footnote{Since $K_{g,d,n}$ is discrete random variable, some authors may use the term probability mass function (PMF) instead of probability density function (PDF).} of $K_{g,d,n}$ conditioned on the minimizers $\hX^\ell$ being irreducible is well approximated by a Gaussian curve.
    \end{obs}

	Gaussian curves are graphs of functions of the form
	$$G(x)=\frac{1}{\sqrt{2 \pi \sigma^2}}e^\frac{(x-\mu)^2}{2\sigma^2}.$$
	We use least square error and weighted least square error to approximate the PDF of our data set with a Gaussian curve.

	We next illustrate Observation \ref{obs:KgdnIsUnimodal} on a few examples. In these examples we consider irreducible minimizers $\hX^\ell$ for randomly generated $(A,\ell)$ pairs where $A \in \smdrg$ is a tuple which defines a bounded irreducible free spectrahedron and $\ell$ is a linear functional defined on $\cD_A (6)$. Here either $g=4$ and $d=5$ or $g=d=5$.

	\subsubsection{Least Square Error Fits} The least square error is defined to be
	$$\sqrt{\frac{1}{|\cK|} \cdot \sum_{k \in \cK} (G(k)-d_k)^2}$$ 
	where $\cK$ is the set of all kernel dimensions we observe in the experiment, $d_k$ is the probability of kernel dimension $k$ in our experiment.

	The raw data listed below has the following form: 
	\[
	\left\{\dim \ker L_A (\hat{X}^\ell), \frac{ \mathrm{number \ of \ occurrences \ of \ kernel \ dimension\ }}{ \mathrm{total \ number \ of\  irreducible\ } \hat{X}^\ell} \right\}.
	\] 
	\begin{enumerate}
    	\item 
    	$g=4$, $d=5$, $n=6$:
    	    $$\{\{5,\frac{2}{9076}\},\{6,\frac{1016}{9076}\},\{7,\frac{5878}{9076}\},\{8,\frac{2145}{9076}\},\{9,\frac{34}{9076}\},\{10,\frac{1}{9076}\}\}$$
    
    	    Gaussian fit (left): $\mu \rightarrow 7.13537, \sigma \rightarrow 0.600874$. The error is 0.000855473.
    	
        \item $g=5$,$d=5$,$n=6:$
            $$\{\{6,\frac{5}{9165}\},\{7,\frac{367}{9165}\},\{8,\frac{4313}{9165}\},\{9,\frac{4062}{9165}\},\{10,\frac{402}{9165}\},\{11,\frac{16}{9165}\}\}$$

        	Gaussian fit (right): $\mu \rightarrow 8.4776, \sigma \rightarrow 0.646746.$ The error is 0.00322601.

	\end{enumerate}
	
        	\begin{center}
        	    \begin{minipage}{0.3\textwidth}
        			\includegraphics[width=\linewidth]{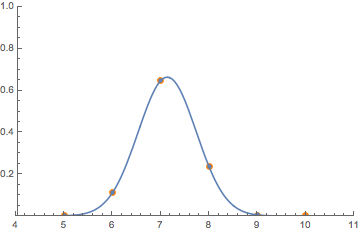} \par \center{$g=4$, $d=5$, $n=6$}
        		\end{minipage}
        		\hspace{0.1\textwidth}
        		\begin{minipage}{0.3\textwidth}
        			\includegraphics[width=\linewidth]{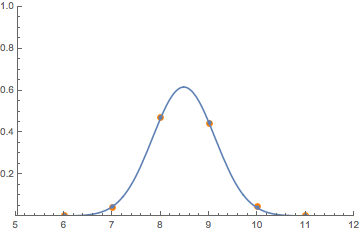} \par \center{$g=5$, $d=5$, $n=6$}
        		\end{minipage}
        	\end{center}
        	
	Here the $x$-axis is the dimension of the kernel of $L_A (\hX^\ell)$, and the $y$-axis is the frequency of that kernel dimension. 

	\subsubsection{Weighted Least Square Error Fits} Since our data set has many points which are close to zero, we also consider the weighted least square error 
	$$\sqrt{\frac{1}{|\cK|} \cdot \sum_{k \in \cK} \left(\frac{G(k)-d_k}{d_k}\right)^2}$$

	The NMinimize function is used to generate the local minimum of the error term. The minimization is initialized by setting the initial parameters $\mu$ to be in the range of our data set with increment 0.01 and $\sigma$ to be the standard deviation of our data.
	
	\begin{enumerate}
	    \item $g=4$, $d=5$, $n=6$:
	        $$\{\{5,\frac{2}{9076}\},\{6,\frac{1016}{9076}\},\{7,\frac{5878}{9076}\},\{8,\frac{2145}{9076}\},\{9,\frac{34}{9076}\},\{10,\frac{1}{9076}\}\}$$

	        Gaussian: $\mu \rightarrow 7.20503, \sigma \rightarrow 0.551157$. The error is 0.438183.
	        
	    \item $g=5$, $d=5$, $n=6$:
	        $$\{\{6,\frac{5}{9165}\},\{7,\frac{367}{9165}\},\{8,\frac{4313}{9165}\},\{9,\frac{4062}{9165}\},\{10,\frac{402}{9165}\},\{11,\frac{16}{9165}\}\}$$

	        Gaussian: $\mu \rightarrow 8.5492, \sigma \rightarrow 0.684021$. The error is 0.257735.
	\end{enumerate}

	  \begin{center}
		\begin{minipage}{0.3\textwidth}
			\includegraphics[width=\linewidth]{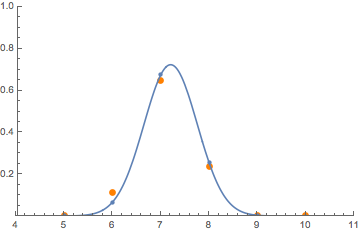} \par \center{$g=4$, $d=5$, $n=6$, (weighted)}
		\end{minipage}
		\hspace{0.1\textwidth}
		\begin{minipage}{0.3\textwidth}
			\includegraphics[width=\linewidth]{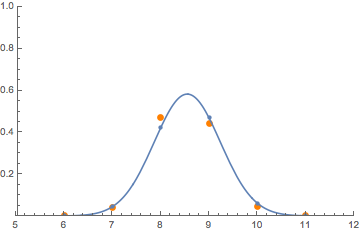} \par \center{$g=5$, $d=5$, $n=6$, (weighted)}
		\end{minipage}
	\end{center}

We do not see a definitive superiority of the accuracy of least square fits versus weighted least square fits. Both seem to give reasonably good approximations.

In addition to trying Gaussian fits, we also tried logistic fits. While the logistic curves fit our data reasonably well, the Gaussian fits were always superior. For this reason we do not discus logistic fits further. 

\ssec{Gaussian fits may fail for $g=2$}  Observation \ref{obs:KgdnIsUnimodal} does not always hold when we take $g=2$. For $g=2$ and large $n$ there are cases where $K_{g,d,n}$ is not well fit by a Gaussian curve with least square error and/or weighted least square error. These cases are shown below.

		\begin{center}
		\begin{minipage}{0.3\textwidth}
			\includegraphics[width=\linewidth]{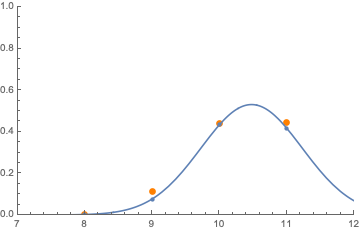} \par \center{$g=2$, $d=3$, $n=12$}
		\end{minipage}
		\hspace{0.1\textwidth}
		\begin{minipage}{0.3\textwidth}
			\includegraphics[width=\linewidth]{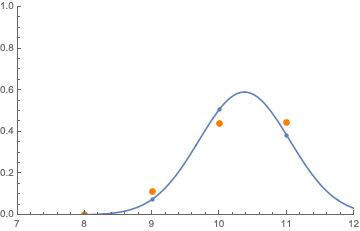} \par \center{$g=2$, $d=3$, $n=12$, (weighted)}
		\end{minipage}
	\end{center}

	\begin{center}
		\begin{minipage}{0.3\textwidth}
			\includegraphics[width=\linewidth]{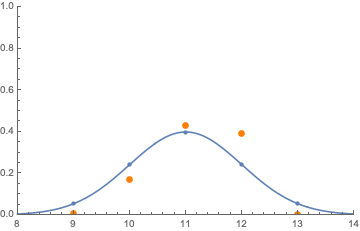} \par \center{$g=2$, $d=3$, $n=13$}
		\end{minipage}
		\hspace{0.1\textwidth}
		\begin{minipage}{0.3\textwidth}
			\includegraphics[width=\linewidth]{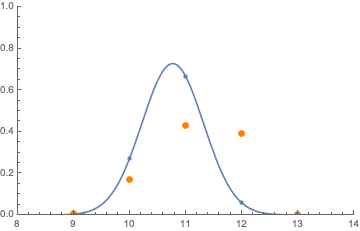} \par \center{$g=2$, $d=3$, $n=13$, (weighted)}
		\end{minipage}
	\end{center}

	\begin{center}
		\begin{minipage}{0.3\textwidth}
			\includegraphics[width=\linewidth]{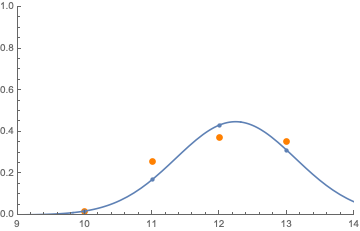} \par \center{$g=2$, $d=3$, $n=14$}
		\end{minipage}
		\hspace{0.1\textwidth}
		\begin{minipage}{0.3\textwidth}
			\includegraphics[width=\linewidth]{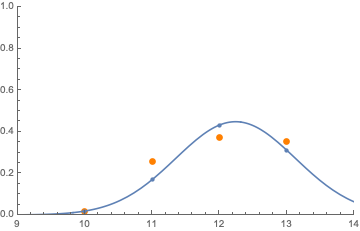} \par \center{$g=2$, $d=3$, $n=14$, (weighted)}
		\end{minipage}
	\end{center}

	\begin{center}
		\begin{minipage}{0.3\textwidth}
			\includegraphics[width=\linewidth]{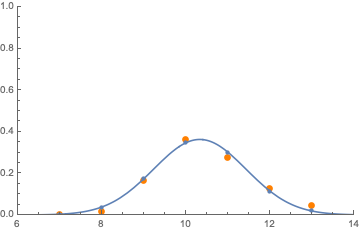} \par \center{$g=2$, $d=5$, $n=14$}
		\end{minipage}
		\hspace{0.1\textwidth}
		\begin{minipage}{0.3\textwidth}
			\includegraphics[width=\linewidth]{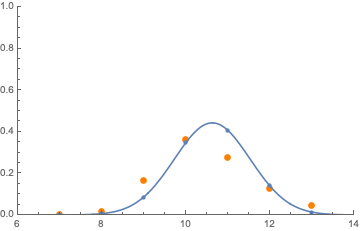} \par \center{$g=2$, $d=5$, $n=14$, (weighted)}
		\end{minipage}
	\end{center}
	
	One may notice that the values for $n$ for which $K_{2,3,n}$ fails to be fit by a Gaussian curve are $n= 12, 13, 14$. These values are on the large side for the typical range of $n$ used in our experiments, namely $n=1-8$. This may lead one to wonder if the failure for $K_{g,d,n}$ to be fit by a Gaussian curve is phenomena which occurs for $g=2$ or a phenomena which occurs for large $n$. However, as shown in Table \ref{table:YiNFEPRC}, for $g=3$ and $d=5,6,7$, we ran experiments on levels $n=1-13$. For these values of $g,d,n$, we found that the distribution of $K_{g,d,n}$ is well fit by a Gaussian curve. Also, for $g=2$, $d=3$, $n=1-8$, there are at most 3 data points, so fitting is moot. These lead us to believe that the failure of $K_{g,d,n}$ to be fit by a Gaussian curve is indeed a $g=2$ phenomena rather than a large $n$ phenomena. 
	
	\section{Reducible vs. Irreducible Extreme Points}
	\label{sec:irredExp}
	
\def\nearcertred{nearly certain that a \RC \  and a \RPT \  linear functional $\ell$ has a minimizer $\hX^\ell$ over a $g,d$-free spectrahedron $\cD_A$} 

	In this section, we shall see that the proportion of reducible optimizers generated using \RC \ linear functionals is monotone non-increasing as $n$ increases; indeed, the proportion of reducible optimizers is well fit by an exponential function. When using \RPT \ linear functionals, we again always observe monotone behaviour in the proportion of reducible extreme points; however, in this setting we observe both monotone non-increasing and monotone non-decreasing behaviour.

Fix an integer $n$ and suppose $\cD_A$ is an irreducible free spectrahedron. For a linear functional $\ell: \smnrg \to \bbR$, let $\hX^\ell$ denote the minimizer of $\ell$ over $\cD_A (n)$ as usual. We let 
	\[
	 p_n(A):=\frac{\# \text{reducible } \hX^\ell}{\# \text{irreducible } \hX^\ell+ \# \text{reducible } \hX^\ell}
	 \index{$p_n(A)$}
	 \]
	 denote the ratio of reducible $\hX^\ell$ generated to the total number of $\hX^\ell$ generated when $\ell$ varies over the collection of random linear functionals chosen in the particular experiments under discussion.

Similarly, for fixed natural numbers $g$ and $d$ we let
	\[
		p_n(g,d):=\frac{\# \text{reducible } \hX^\ell}{\# \text{irreducible } \hX^\ell+ \# 		\text{reducible } \hX^\ell}
		\index{$p_n(g,d)$}
	\]  
denote the ratio of reducible $\hX^\ell$ generated to the total number of $\hX^\ell$ generated when the pair $(A,\ell)$ varies over a collection of pairs consisting of an irreducible defining tuples $A \in \smdrg$ for bounded free spectrahedra with a \RC \ or \RPT \ linear functional $\ell$. 

We briefly note that the case $n=1$ is omitted in the following discussion as for any $d$ and $g$ and for any $A \in \smdrg$ one always has $p_1 (g,d)=0$ and $p_1 (A)=0$ since a tuple of real numbers is always irreducible. 
	
	\subsection{RC Behaviour}
	In this subsection we present experimental results for $p_n(A)$ and $p_n(g,d)$ found when using RC linear functionals. 
	
	\begin{obs} In all of our experiments, excluding $g=d=3$, using \RC \ linear functionals we observe the following.
		\label{obs:RCLFRedProp}
		\begin{enumerate}
			
			\item 
			\label{it:pngdBehavior}
			If $g \le d$, then it is \ncr\ that $p_n(g,d)$ is monotone decreasing as $n$ increases and 
			also typically one has 
			$$p_n(g,d) \to 0 \text{ as } n \to \infty.$$

			\item
			\label{it:pnABehavior}
			 If $A \in \smdrg$ with $3 \le g \le d$, then it is \ncr\ that $p_n(A)$ is monotone decreasing as $n$ increases and also typically one has 
			$$p_n(A) \to 0 \text{ as } n \to \infty.$$ 
		
		\end{enumerate}
		
	\end{obs}
	
	\evidence For Observation \ref{obs:RCLFRedProp} \eqref{it:pngdBehavior}, $p_n(g,d)$ was computed for all values of $g,d,n$ that occur in Table \eqref{table:YiNFEPRC} and \eqref{table:YiNFEPRCg2}. For Observation \ref{obs:RCLFRedProp} \eqref{it:pnABehavior}, $p_n(A)$ was computed for about $30$ different free spectrahedra. $p_n (A)$ was found to be non-decreasing on all but one of these free spectrahedra. The one anomaly example had an abnormally large number of numerically ill-conditioned points. On this spectrahedron, 4318/5000 points were found to numerically ill-conditioned at level $n=8$. 
		 \qed

	Note: The reason we say it is nearly certain rather than certain is because we occasionally observe a very small deviation from monotonicity. Details (mostly about item \ref{it:pngdBehavior}) are given in the following subsubsections.

	\sssec{Small deviations from monotone behaviour can occur in $p_n (g,d)$}
	
	When using \RC \ linear functionals  in one experiment, we observe a slight increase in $p_n (3,6)$ as $n$ increases from $10$ to $11$ and as $n$ increases from from $11$ to $12$. The following table lists experimentally observed values of $p_n (3,6)$ for $2 \leq n \leq 13$. For each $n$, a total of $10,000$ pairs $(A,\ell)$ were generated.
	
       {\footnotesize{
    \begin{table}[h]
    		\begin{tabular}{|c|c|c|c|c|c|c|c|c|c|c|c|c|}
    			\hline
    			 2 & 3 & 4 & 5 & 6 & 7 & 8 & 9 & 10 & 11 & 12 & 13 \\ \hline
    			 15.3\% & 9.74\% & 5.48\% & 2.82\% & 1.92\% & 1.11\% & 0.880\% & 0.594\% & 0.376\% & 0.389\% & 0.392\% & 0.192\% \\ \hline
    		\end{tabular}
    		\caption{Observed values for $p_2(3,6)$ through $p_{13}(3,6)$.}
	\end{table}
    }}

	To more closely examine cases where monotonicity failed,  we looked at $p_n (2,5)$, where we observed a slight increase from 1.78\% to 2.18\% when level goes from 12 to 13. We did 4 experiments, with different seeds and each has 10000 cases on each level.

	     \begin{table}[h]
	 	\begin{tabular}{|c|c|c|}
	 		\hline
	 		Exp. No. & $p_{12} (2,5)$ & $p_{13} (2,5)$ \\ \hline
	 		1 & 1.78\% & 2.18\% \\ \hline
	 		2 & 1.81\% & 0.29\% \\ \hline
	 		3 & 1.66\% & 0.47\% \\ \hline
	 		4 & 1.77\% & 0.40\% \\ \hline
	 	\end{tabular}
	 	\caption{Four experiments on $p_{12} (2,5)$ and $p_{13} (2,5)$}
	 \end{table}
 
 	The percentage has a smaller variance at level $12$ than at level $13$. And by averaging the four experiments we found $p_{12} (2,5) =1.76\%$, and $p_{13} (2,5) =0.84\%$, which is indeed a decrease.

	\sssec{Exponential Fit} We now see that $p_n(g,d)$ and $p_n(A)$ are well approximated by exponentially decreasing functions. 
	
	\begin{obs}
	\label{obs:RCLFExpFit}
	For \RC \ linear functionals, if $n \ge 2$ and if $2 \le g < d$ or $g=d=4 \mathrm{ \ or \ } 5$, then the ratios $p_n (g,d)$ and $p_n (A)$ are reasonably approximated by a decreasing exponential function of the form $$f(n)=ae^{-rn}$$
	where $r$ is some positive constant depending on $g$ and $d$. Furthermore, $r$ increases with $d$. This is illustrated in the following figures.
	\end{obs}
	
	In the following figures the $x$-axis is the level $n$ of the free spectrahedra over which RC linear functionals were optimized, and the $y$-axis is the proportion of optimizers $\hat{X}$ which were reducible.

\begin{center}
	\begin{minipage}{0.3\textwidth}
		\includegraphics[width=\linewidth]{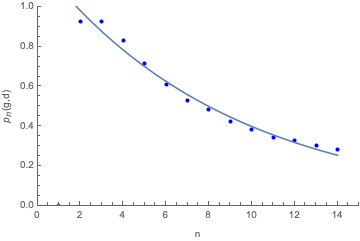}\par
		\center{ $p_n (2,3)$ \ \ RC}
	\end{minipage}
	\begin{minipage}{0.3\textwidth}
		\includegraphics[width=\linewidth]{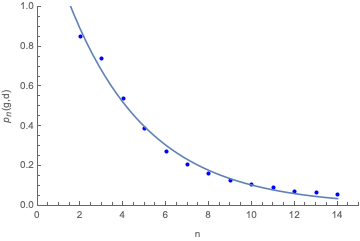}\par
		\center{$p_n (2,4)$ \ \ RC}
	\end{minipage}
	\begin{minipage}{0.3\textwidth}
		\includegraphics[width=\linewidth]{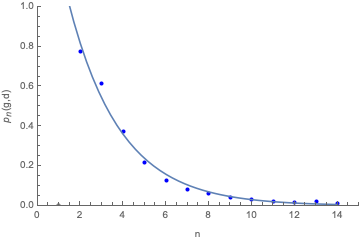}\par
			\center{$p_n (2,5)$ \ \ RC}
	\end{minipage}
\end{center}

\begin{center}
		\begin{minipage}{0.3\textwidth}
		\includegraphics[width=\linewidth]{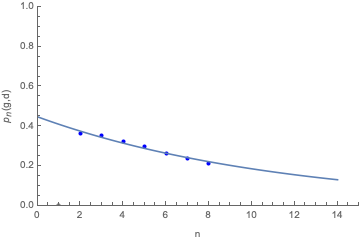}\par
		\center{ $p_n (3,4)$ \ \ RC} 
	\end{minipage}
	\begin{minipage}{0.3\textwidth}
		\includegraphics[width=\linewidth]{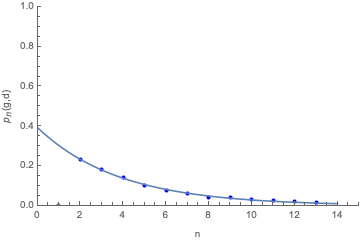}\par
		\center{ $p_n (3,5)$ \ \ RC} 
	\end{minipage}
\end{center}

\begin{center}

	\begin{minipage}{0.3\textwidth}
		\includegraphics[width=\linewidth]{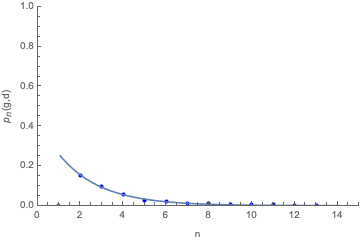}\par
		\center{ $p_n (3,6)$ \ \ RC}
	\end{minipage}
	\begin{minipage}{0.3\textwidth}
	\includegraphics[width=\linewidth]{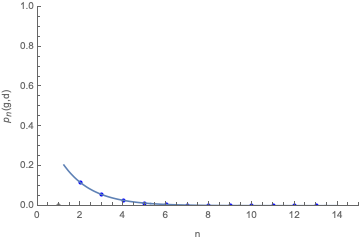}\par
		\center{ $p_n (3,7)$ \ \ RC} 
\end{minipage}
\end{center}

\begin{center}
	\begin{minipage}{0.3\textwidth}
		\includegraphics[width=\linewidth]{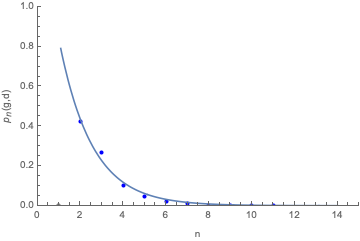}\par
		\center{ $p_n (4,4)$ \ \ RC}
	\end{minipage}
	\begin{minipage}{0.3\textwidth}
		\includegraphics[width=\linewidth]{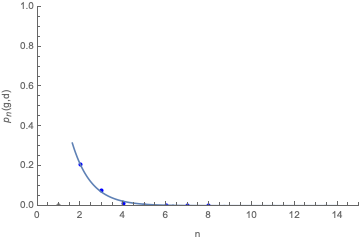}\par
		\center{ $p_n (4,5)$ \ \ RC} 
	\end{minipage}
	\begin{minipage}{0.3\textwidth}
		\includegraphics[width=\linewidth]{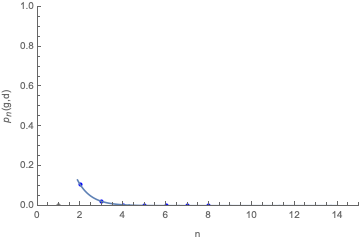}\par
		\center{ $p_n (4,6)$ \ \ RC} 
	\end{minipage}
\end{center}

\begin{center}
	\begin{minipage}{0.3\textwidth}
		\includegraphics[width=\linewidth]{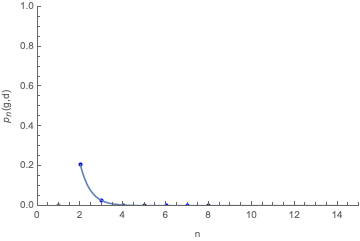}\par
		\center{ $p_n (5,5)$ \ \ RC} 
	\end{minipage}
	\begin{minipage}{0.3\textwidth}
		\includegraphics[width=\linewidth]{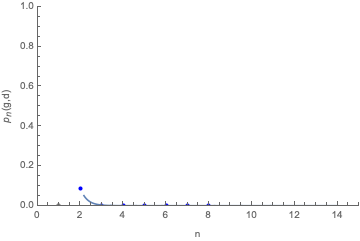}\par
		\center{ $p_n (5,6)$ \ \ RC}
	\end{minipage}
	\begin{minipage}{0.3\textwidth}
		\includegraphics[width=\linewidth]{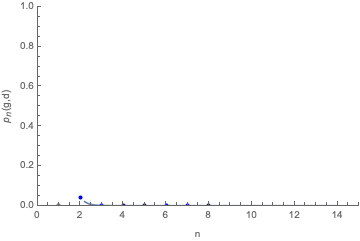}\par
		\center{ $p_n (5,7)$ \ \ RC}
	\end{minipage}
\end{center}
\begin{center}

	\begin{minipage}{0.3\textwidth}
		\includegraphics[width=\linewidth]{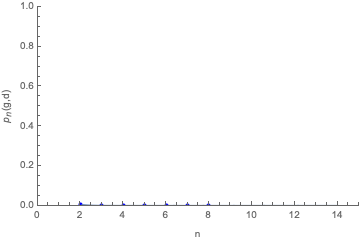}\par
		\center{ $p_n (6,7)$ \ \ RC}
	\end{minipage}
\end{center}

Similar to $p_n(g,d)$, we find that $p_n (A)$ is well approximated by decreasing exponential curves when using RC linear functionals. However, we do not observe a clear pattern of $r$ versus $d$ for $p_n(A)$ as in Observation \ref{obs:RCLFExpFit}.

\subsection{RPT behaviour}

	In experiments using RPT linear functionals, for $g > 3$, we observe that $p_n (A)$ and $p_n (g,d)$ have monotonically (exponentially) decreasing behaviour similar to what we saw using RC linear functionals. On the other hand, for $g=3$, $p_n (A)$ and $p_n (g,d)$ do not show an obvious decreasing trend; and for $g=2$, $p_n (g,d)$ are decreasing slower than exponential decay.

	\sssec{$p_n (g,d)$} The following graphs show $p_n (g,d)$ using RPT linear functional experiments. 

	\begin{center}
		\begin{minipage}{0.3\textwidth}
			\includegraphics[width=\linewidth]{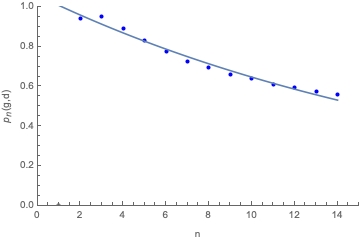}\par
			\center{ $p_n (2,3)$ \ \ RPT}
		\end{minipage}
		\begin{minipage}{0.3\textwidth}
			\includegraphics[width=\linewidth]{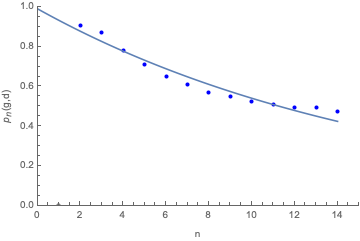}\par
			\center{$p_n (2,4)$ \ \ RPT}
		\end{minipage}
		\begin{minipage}{0.3\textwidth}
			\includegraphics[width=\linewidth]{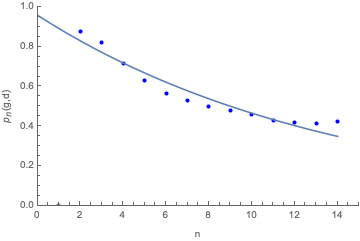}\par
			\center{$p_n (2,5)$ \ \ RPT}
		\end{minipage}
	\end{center}
	
	\begin{center}
		\begin{minipage}{0.3\textwidth}
			\includegraphics[width=\linewidth]{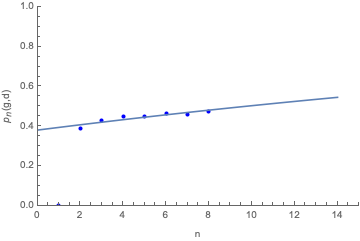}\par
			\center{ $p_n (3,4)$ \ \ RPT} 
		\end{minipage}
		\begin{minipage}{0.3\textwidth}
			\includegraphics[width=\linewidth]{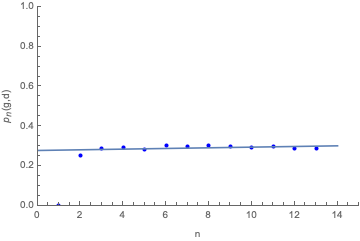}\par
			\center{ $p_n (3,5)$ \ \ RPT} 
		\end{minipage}
	\end{center}
	
	\begin{center}
		
		\begin{minipage}{0.3\textwidth}
			\includegraphics[width=\linewidth]{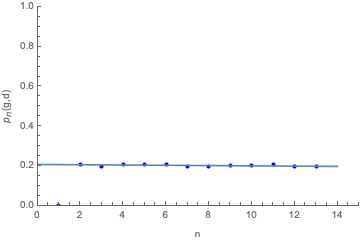}\par
			\center{ $p_n (3,6)$ \ \ RPT}
		\end{minipage}
		\begin{minipage}{0.3\textwidth}
			\includegraphics[width=\linewidth]{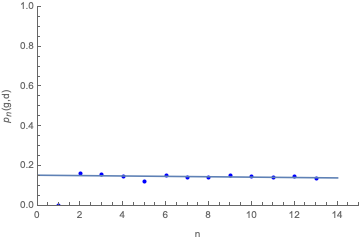}\par
			\center{ $p_n (3,7)$ \ \ RPT}
		\end{minipage}
	\end{center}

	\begin{center}
		\begin{minipage}{0.3\textwidth}
			\includegraphics[width=\linewidth]{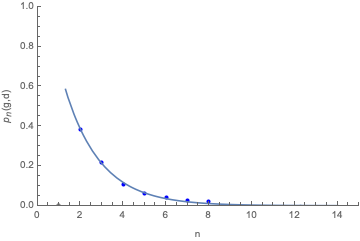}\par
			\center{ $p_n (4,4)$ \ \ RPT} 
		\end{minipage}
		\begin{minipage}{0.3\textwidth}
			\includegraphics[width=\linewidth]{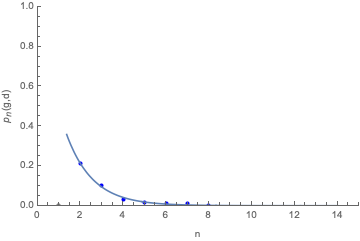}\par
			\center{ $p_n (4,5)$ \ \ RPT}
		\end{minipage}
		\begin{minipage}{0.3\textwidth}
			\includegraphics[width=\linewidth]{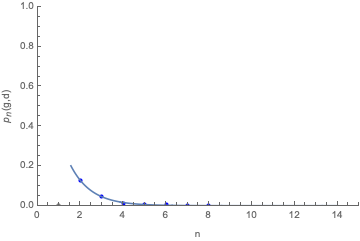}\par
			\center{ $p_n (4,6)$ \ \ RPT} 
		\end{minipage}
	\end{center}
	
	\begin{center}
		\begin{minipage}{0.3\textwidth}
			\includegraphics[width=\linewidth]{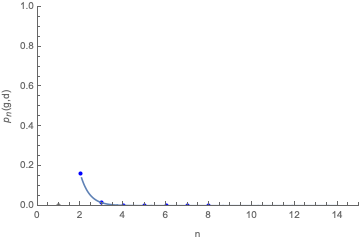}\par
			\center{ $p_n (5,5)$ \ \ RPT} 
		\end{minipage}
		\begin{minipage}{0.3\textwidth}
			\includegraphics[width=\linewidth]{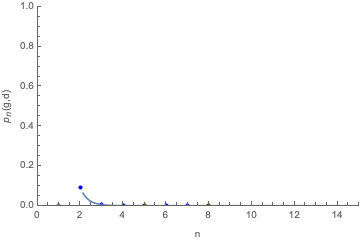}\par
			\center{ $p_n (5,6)$ \ \ RPT}
		\end{minipage}
		\begin{minipage}{0.3\textwidth}
			\includegraphics[width=\linewidth]{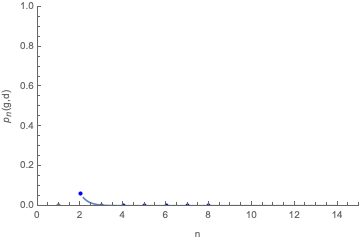}\par
			\center{ $p_n (5,7)$ \ \ RPT} 
		\end{minipage}
	\end{center}
	\begin{center}
		
		\begin{minipage}{0.3\textwidth}
			\includegraphics[width=\linewidth]{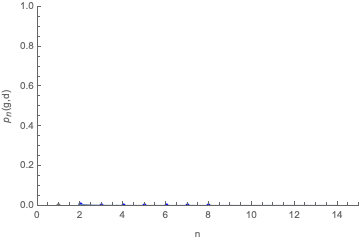}\par
			\center{ $p_n (6,7)$ \ \ RPT} 
		\end{minipage}
	\end{center}

	\sssec{$p_n (A)$} The subsection gives graphs representative of $p_n (A)$ behaviours seen in our experiments using RPT linear functionals when $g = 3,4,5$.
	
	For $g = 4,5$, the behaviour for $p_n (A)$ is consistent with that of $p_n (g,d)$ when using RPT linear functionals; the graphs decrease exponentially.

	\begin{center}
		\begin{minipage}{0.3\textwidth}
			\includegraphics[width=\linewidth]{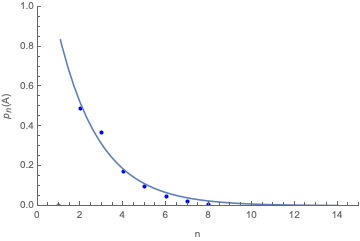}
			\par
			\center{  g4d4sphd4 \ $p_n (A)$ \  RPT } 
		\end{minipage}
		\hspace{0.1\textwidth}
		\begin{minipage}{0.3\textwidth}
			\includegraphics[width=\linewidth]{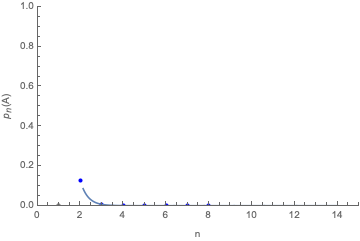}
			\par
			\center{  g5d6sphd4 \ $p_n (A)$ \  RPT }
		\end{minipage}
	\end{center}
    
    For $g = 3$, since we do not observe a decreasing trend for $p_n (g,d)$ using RPT linear functionals, we further examine $p_n (A)$ on single spectrahedra. We observe that $p_n (A)$ has heterogeneous behaviours. The following graphs are representative of behaviours seen in our $g = 3$ experiments using RPT linear functionals. The graphs are ordered in frequency of occurrence of the represented behaviours. 
    
    \begin{center}
    	\hspace{0.02\textwidth}
    	\begin{minipage}{0.3\textwidth}
    		\includegraphics[width=\linewidth]{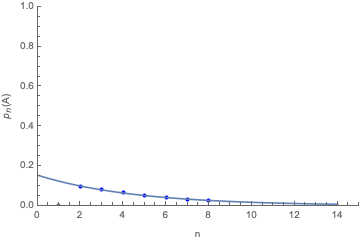}
    		\par
    		\center{  g3d5sphd4  \ $p_n (A)$ \  RPT  \\
    			decreasing to 0\\ 
    			frequency = 9/18}
    	\end{minipage}
    	\hspace{0.1\textwidth}
    	\begin{minipage}{0.3\textwidth}
    		\includegraphics[width=\linewidth]{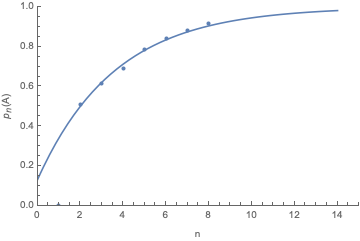}
    		\par
    		\center{g3d5sphd2  \ $p_n (A)$ \  RPT \\ increasing to 1 (asymptotically), frequency = 4/18}
    	\end{minipage}
    \end{center}
    \begin{center}
    	\begin{minipage}{0.3\textwidth}
    		\includegraphics[width=\linewidth]{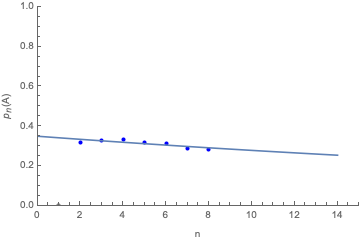}
    		\par
    		\center{g3d4sphd5 \ $p_n (A)$ \  RPT \\ deceasing but maybe not to 0, frequency = 3/18}
    	\end{minipage}
    	\hspace{0.1\textwidth}
    	\begin{minipage}{0.3\textwidth}
    		\includegraphics[width=\linewidth]{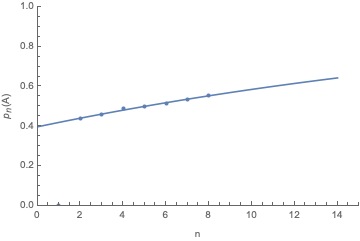}
    		\par
    		\center{g3d5sphd5 \ $p_n (A)$ \  RPT \\ increasing but maybe not to 1, frequency = 1/18}
    	\end{minipage}
    \end{center}
    
    \begin{center}
    	\begin{minipage}{0.3\textwidth}
    		\includegraphics[width=\linewidth]{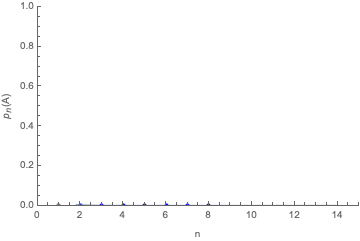}
    		\par
    		\center{g3d5sphd6 \ $p_n (A)$ \  RPT \\ decreasing to 0, frequency = 1/18}
    	\end{minipage}
    \end{center}

    \subsection{Irreducibility of extreme points can depend heavily on the spectrahedron}
    
    We end the section by briefly noting that there exist known spectrahedra with exceptional properties in terms of reducibility of extreme points. For example, if $A$ is a $g$ tuples of $(g+1) \times (g+1)$ diagonal matrices which define a bounded free spectrahedron, then $D_A$ is called a free simplex in $g$ variables. \cite[Theorem 6.5]{EHKM18} shows that free extreme points of a free simplex are exactly equal to the Euclidean extreme points of $\cD_A (1)$. That is, for $n\geq 2$, for a free simplex every Arveson extreme point of $\cD_A (n)$ is fully reducible in the sense that they are simultaenously diagonalizable. Also see \cite{FNT17} and \cite{PSS18} for further discussion of free simplices. 
    
     It also is possible to choose $A$ so that $\cD_A$ is an irreducible free spectrahedron whose free extreme points are exactly equal to the Euclidean extreme points at level one of $\cD_A$. In particular, if one takes
    \[
    A_1= \begin{pmatrix} 0 & 1 \\
    1 & 0 
    \end{pmatrix}
    \mathand
    A_2 = \begin{pmatrix} 1 & 0 \\
    0 & -1 
    \end{pmatrix}
    \]
    Then $A=(A_1,A_2)$ defines a free disc which has the property $\free \cD_A = \euc \cD_A (1)$, see \cite[Proposition 7.5]{EHKM18}.  
    
 A free spectrahedron which has all its free extreme points at level one is an example of a minimal matrix convex set, see \cite{PSS18}. As such, both these examples can be viewed as special cases of \cite[Proposition 2.5]{PS19}.
    
    \sec{Software and data availability}
    \label{sec:softwareanddata}

	\begin{sloppypar}

    The NCSE package created for these experiments as well as the raw data may be found online in directories \url{https://github.com/NCAlgebra/UserNCNotebooks/tree/master/NCSpectrahedronExtreme}, and  \url{https://github.com/NCAlgebra/UserNCNotebooks/tree/master/EvertFuHeltonYin}, what we now describe as NCSpectrahedronExtreme and EvertFuHeltonYin respectively.

  \end{sloppypar}

    \ssec{Software}

    Our experiments are all run using the NCSE \cite{EOYH19} package for NCAlgebra \cite{OHMS17}, which, at the time of running our experiments\footnote{NCSE has since been updated to use the Mathematica 12 SDP by default. As an option, a user may still use the NCAlgebra SDP.}, used the NCAlgebra SDP package. Experiments were mostly run in Mathematica 11. Mathematica 12 has a semidefinite program embedded, so we compared some of our results to results obtained using the Mathematica SDP when it became available. We found that the choice of SDP solver had little impact on the outcome.

    \ssec{Data availability and reproduction}
    
In EvertFuHeltonYin, there are two folders, one for experiments on fixed spectrahedra and one for experiments on collections of randomly generated spectrahedra. Each folder contains a collection of spread sheets, with each spread sheet containing all data for runs on spectrahedra with fixed $g$ and $d$ using either RPT or RC linear functionals.
    
    In addition each folder contains a Mathematica notebook which may be used in combination with NCSE to reproduce our experiments.

In Section \S\ \ref{sec:irredExp}, we have a few figure labels related to the raw data. There is a naming system in which each figure corresponds to a ``sheet". For example, the figure ``$p_n(2,4)$ RC" corresponds to the sheet ``g2d4" in ``random\_sphd.xlsx";  the figure ``$p_n(2,4)$ RPT" corresponds to the sheet ``g2d4" in ``random\_sphd\_RPT.xlsx". The figure ``g4d4sphd4 $p_n (A)$  RPT" corresponds to the sheet ``irredg4d4John4 RPT" inside the file "irredg4d4\_RPT.xlsx".

	\newpage
	
\centerline{	NOT FOR PUBLICATION}
		
		\tableofcontents

	\printindex
	
\end{document}